\documentclass[a4paper, 11pt,reqno]{amsart}
\usepackage[english]{babel}
\usepackage[utf8]{inputenc}
\usepackage{amsthm, amsmath, amssymb, amsrefs, enumerate, enumitem, mathtools, bbm, mathrsfs}
\usepackage{array, listings, tabularx}

\usepackage[T1]{fontenc}
\usepackage{lmodern, microtype}
\usepackage{a4wide}
\binoppenalty=\maxdimen
\relpenalty=\maxdimen

\usepackage{hyperref}

\theoremstyle{plain}
\newtheorem{thm}{Theorem}[section]
\newtheorem*{thm*}{Theorem}
\newtheorem{prop}[thm]{Proposition}
\newtheorem{lemma}[thm]{Lemma}
\newtheorem{cor}[thm]{Corollary}
\theoremstyle{definition}

\theoremstyle{remark}
\newtheorem*{rmk}{Remark}
\newtheorem*{rmks}{Remarks}

\newcommand{\C}{\mathbb{C}}
\renewcommand{\H}{\mathbb{H}}
\newcommand{\Z}{\mathbb{Z}}
\newcommand{\T}{\mathbb{T}}
\newcommand{\Q}{\mathbb{Q}}
\newcommand{\N}{\mathbb{N}}
\newcommand{\R}{\mathbb{R}}
\newcommand{\slz}{{\text {\rm SL}}_2(\mathbb{Z})}
\newcommand{\re}{\textnormal{Re}}
\newcommand{\im}{\textnormal{Im}}

\newcommand{\vt}[1]{\left\lvert #1 \right\rvert}
\newcommand{\id}{\mathrm{id}}
\newcommand{\dm}{\mathrm{d}}
\newcommand{\Ec}{\mathcal{E}}

\newcommand{\Qc}{\mathcal{Q}}
\newcommand{\Fc}{\mathcal{F}}
\newcommand{\Ps}{\mathscr{P}}

\setlist{nosep}
\setlist{noitemsep}
\numberwithin{equation}{section}
\lstdefinelanguage{Sage}[]{Python}
{morekeywords={False,sage,True},sensitive=true}
\newcolumntype{M}[1]{>{\centering\arraybackslash}m{#1}}

\title[Central $L$-values and local polynomials]{Central $L$-values of newforms and local polynomials}
\author{Joshua Males}
\address{\textnormal{(unaffiliated)}}
\email{joshua.males@bristol.ac.uk}
\author{Andreas Mono}
\address{\textnormal{\textit{Current address:}} Department of Mathematics, Vanderbilt University, 1326 Stevenson Center, Nashville, TN 37240, USA \newline \indent \textnormal{\textit{Former address:}} Department of Mathematics and Computer Science, Division of Mathematics, University of Cologne, Weyertal 86--90, 50931 Cologne, Germany}
\email{andreas.mono@vanderbilt.edu}
\author{Larry Rolen}
\address{Department of Mathematics, Vanderbilt University, 1326 Stevenson Center, Nashville, TN 37240, USA}
\email{larry.rolen@vanderbilt.edu}
\author{Ian Wagner}
\address{\textnormal{(unaffiliated)}}
\email{ianwagner11@gmail.com}

\begin{document}

\begin{abstract}
In this paper, we characterize the vanishing of twisted central $L$-values attached to newforms of square-free level in terms of certain polynomials of quadratic forms introduced by Zagier and the action of finitely many Hecke operators thereon. To be more precise, we establish that a twisted central $L$-value attached to a newform vanishes if and only if a certain explicitly computable polynomial is constant. We describe these constants explicitly in two different ways. One of the descriptions involves the generalized Hurwitz class numbers, which were introduced by Pei and Wang in $2003$. We provide some numerical examples and conclude by offering some questions for future work.
\end{abstract}



\maketitle

\section{Introduction and statement of results}

\subsection{Motivation and previous work}

In his study of the Doi--Naganuma lift from modular forms to Hilbert modular forms \cite{zagier75}, Zagier defined the special functions 
\[
f_{k,D}(z)\coloneqq\sum_{[a,b,c]\in\mathcal Q_D}\frac{1}{(az^2+bz+c)^k}
.
\]
Here, for any $2\leq k\in\N$ and any discriminant $D>0$, $\mathcal Q_D$ is the set of integral binary quadratic forms of discriminant $D$, and the corresponding $f_{k,D}$ is modular on $\slz$ with weight $2k$. These functions are also Poincar\'e series \cite{IOS}, but with respect to Petersson's {hyperbolic expansions} of modular forms (see also Katok's deep study of such functions \cite{katok}). 
The significance of these forms was later amplified by the work of Kohnen \cite{Kohnen} and Kohnen--Zagier \cites{koza81, koza84}. To describe this, we need a slightly decorated generalization.
For discriminants $D, D_0 \equiv 0, 1 \pmod{4}$ and $DD_0 >0$, define
\begin{equation*}
f_{k,N,D,D_0}(z) \coloneqq \sum_{Q \in \mathcal{Q}_{N,DD_0}} \chi_{D_0}(Q) Q(z,1)^{-k},
\end{equation*}
where $\mathcal{Q}_{N,DD_0} \coloneqq \{[a,b,c] \in \mathcal{Q}_{DD_0} \colon N\mid a\}$ and $\chi_{D_0}$ is the so-called {\it genus character} from \cite{grokoza}.
For $k\geq 2$, $f_{k,N,D,D_0}$ is a cusp form in $S_{2k}(N)$. (This remains true for $k=1$ and $N$ cubefree, see Kohnen \cite{Kohnen}.)

Kohnen showed that a certain two-variable generating function of these forms has a representation in terms of classical half-integral weight Poincar\'e series as a function in the other variable. Thus, the two-variable generating function is modular in both variables, but with different weights. Kohnen used this observation to produce a kernel function of the famous Shimura \cite{shim} and Shintani \cite{shin} lifts. 

Kohnen and Zagier used this connection to study central $L$-values of modular forms. Specifically, this allowed for a more explicit interpretation of very general work of Waldspurger \cite{waldspurger}. For a suitable Hecke eigenform $f$ of even integral weight $2k$, one can detect the vanishing of, and give exact formulas for, twisted central $L$-values of $f$. Given such an $f$, there exists a cusp form 
$
g(z)=\sum_{n\geq1}a_g(n)q^n, \ q \coloneqq e^{2 \pi i z},
$
of weight $k+1/2$ such that for fundamental discriminants $D$ with $(-1)^k D > 0$, the Fourier coefficients of $g$ at index $D$ gives the square root of the central twisted $L$-value of $f$ up to a non-zero constant (see \cite{Kohnen}*{Corollary 1}:
\begin{align*}
a_g(\vt{D})^2 = C(f,g,k,N,D) L(f\otimes\chi_{D},k).
\end{align*}

Central $L$-values have a long history, tied to essential problems in number theory and arithmetic geometry, particularly thanks to the Birch and Swinnerton--Dyer Conjecture, Bloch--Kato Conjecture, and others. The work of Kohnen and Zagier allows one to detect whether a twisted central $L$-value vanishes, in the case of elliptic curves, whether the curve has (assuming BSD) a point of infinite order. The best known progress towards BSD, which allows one to determine the vanishing of central twisted first $L$-derivatives of elliptic curves, was given by studying heights of Heegner points by Gross, Kohnen, and Zagier \cite{grokoza}.

This was reinterpreted by Bruinier and Ono \cite{BruOnoAnnals} in the context of {\it harmonic Maass forms}. Essentially, these are functions on the upper half-plane which transform like modular forms, but instead of being holomorphic, are merely required to be in the kernel of a weighted Laplacian operator. This forces them to be real-analytic, and to split into two parts. These are the {\it holomorphic part}, which has an ordinary $q$-series expansion, and the {\it non-holomorphic part}, which has an expansion in terms of incomplete gamma functions. What Bruinier and Ono showed is that given a weight 2 newform attached to a rational elliptic curve $E$, there is a harmonic Maass form of weight $1/2$ whose holomorphic part coefficients determine the vanishing of twisted central $L$-derivatives of $E$, and whose non-holomorphic part coefficients determine the vanishing of its twisted central $L$-values.

A key aspect of the theory of harmonic Maass forms is the action of various differential operators. Particularly important is the $\xi_k$ operator of Bruinier and Funke \cite{BruinierFunke}, given by 
\begin{align} \label{eq:xidef}
f \coloneqq \xi_k(F)\coloneqq2i\im(z)^k\overline{\frac{\partial F}{\partial \overline{z}}}.
\end{align}
This is essentially a different normalization of the classical Maass lowering operator. The salient features of this operator is that it maps harmonic Maass forms of negative weight $k$ to holomorphic cusp forms\footnote{This is no longer true if the non-holomorphic part is of linear exponential growth towards the cusps as well, an example is the Maass--Eisenstein series from \cite{thebook}*{Theorem 6.15}.} of weight $2-k$, and that it does so {\it surjectively} (as shown by Bruinier and Funke). Subsequently, Bruinier, Ono and Rhoades \cite{BOR} proved that the $\xi_k$ operator has a ``holomorphic companion'' in the case of negative integral weights $k$, namely the Bol operator 
\begin{align} \label{eq:boldef}
\left(\frac{1}{2\pi i}\frac{\partial}{\partial z}\right)^{1-k},
\end{align}
which maps $F$ to a weakly holomorphic modular form $g$, also of weight $2-k$. Both differential operators admit an inverse operator, namely the {\it holomorphic} resp.\ {\it non-holomorphic Eichler integral} $\mathcal{E}_g$ resp.\ $f^*$ (see \eqref{eq:Eichlerholo}, \eqref{eq:Eichlernonholo}). In other words, one can express the splitting of $F$ into a holomorphic and a non-holomorphic part by writing, up to an additive constant term,
\begin{equation} \label{HMFDecomp}
F = \mathcal E_g+f^*.
\end{equation}

Many of the biggest applications in the theory of harmonic Maass forms revolve around the question of finding ``good'' lifts under $\xi_{k}$ of given cusp forms, and given the important role played by Zagier's $f_{k,D}$ functions, it is natural to search for explicit lifts for them. It is very common in the theory of harmonic Maass forms to decompose forms in terms of a basis of the so-called Maass--Poincar\'e series, which are the canonical lifts of classical cuspidal Poincar\'e series.
 As mentioned above, Zagier's functions are in fact hyperbolic Poincar\'e series, and so describing them in terms of the alternative basis of (elliptic) Maass--Poincar\'e series is not only difficult but also unnatural. Since this realization as Poincar\'e series expresses them as group averages of a {\it seed} hit with the Petersson slash action (see \eqref{eq:slashop}), and since the $\xi_k$ operator intertwines with the slash actions in weights $k$ and $2-k$, the most natural candidate for a lift is a group average of a preimage of the seed itself under $\xi_k$. This leads to a differential equation, which Bringmann, Kane, and Kohnen \cite{bkk} solved and then used to build a new hyperbolic Poincar\'e series. 

This construction inevitably led to a new phenomenon, which these authors named {\it locally harmonic Maass forms}. The idea is that they have the same basic features as harmonic Maass forms, but they have jump discontinuities on hyperbolic geodesics dictated by the quadratic forms of a given discriminant. Similar local discontinuities have also been discovered by Hövel \cite{hoevel} in weight $0$, and by Zagier, as he related in private conversations with the third author. These functions of Bringmann, Kane, and Kohnen led to another derivation of modular properties of special modular integrals of Duke--Imamo\={g}lu--T\'{o}th \cite{DIT}. By investigating a speculation of Duke-Imamo\={g}lu-T\'{o}th in a follow-up paper \cite{duimto10}*{(16)}, similar functions with local discontinuities are the result of analytically continuing Parson's \cite{parson} modular integral in weight $2$ \cites{mo1, mat23} as well as of completing those to modular objects (``local cusp forms'') in even weights $2 < k \equiv 2 \pmod{4}$ \cite{mo2}. Recently, Bringmann and the second author \cite{brimo1} extended the picture by constructing forms with continuously, but not differentially removable singularities, and the first two authors \cites{mamo, mamo2} constructed a family of vector-valued local Maass forms along the lines of H\"ovel.

\subsection{Previous work: discussion and example}
\label{PreviousWorkSection}

These locally harmonic lifts of $f_{k,N,D,D_0}$, denoted by $\mathcal{F}_{1-k,N,D,D_0}$, were used by Ehlen, Guerzhoy, Kane, and the third author in \cite{egkr} to give a new criterion for the vanishing of twisted central $L$-values. This work focused on the finitely many cases of weight $2$ newforms in one-dimensional cusp form spaces. Although this choice was simpler in some ways, it had technical complications due to the levels involved and due to weight $2$ being at the boundary of convergence (so the expressions for the functions above are not absolutely convergent in this case). However, that case was chosen to include the primary motivating example of the {\it congruent number problem}. 

This paper develops an analogue in more general spaces of cusp forms. To illustrate the principle, we first discuss the case of the congruent number problem studied in \cite{egkr}. Recall that a natural number $n$ is {\it congruent} if it is the area of a right triangle with rational side lengths. This problem, studied since antiquity, was famously solved (assuming BSD) by Tunnell in \cite{tunnell}. An elementary argument shows that $n$ is congruent if and only if the $n$-th quadratic twist of the congruent number curve $E\colon y^2=x^3-x$ has a point of infinite order. Under BSD, this point exists if and only if we have $L(f\otimes\chi_n,1)=0$, where $\chi_n=\left(\frac{n}{\cdot}\right)$ and where $f$ is the unique normalized cusp form in $S_{2}(\Gamma_0(32))$. As the space of cusp forms of weight $2$ on $\Gamma_0(32)$ is one-dimensional, if we pick any $D_0$ for which $L(f\otimes \chi_{D_0},1)\neq0$ then (extensions of) Kohnen's work should allow us to conclude 
\begin{equation} \label{Lvanishfvanish}
L(f\otimes\chi_D,1)=0\iff f_{1,32,D,D_0}=0,
\end{equation}
because $f_{1,32,D,D_0}$ must be some multiple of $f$. 

The theory of locally harmonic Maass forms is well-suited to detect the vanishing of $f_{1,32,D,D_0}$. Using its locally harmonic lift $\mathcal F_{0,32,D,D_0}$, it decomposes into three pieces, namely a holomorphic part, a non-holomorphic part, and a so-called local polynomial $P_{k,32,D,D_0}$ (see \eqref{locPolydefn}). In other words, we have
\begin{equation} \label{Fdecomp3pieces}
\mathcal F_{0,32,D,D_0}(z)= \alpha \mathcal E_{f_{1,32,D,D_0}}(z)+\beta f^*_{1,32,D,D_0}(z)+P_{k,32,D,D_0}(z),
\end{equation}
where $\alpha, \beta$ are constants depending on the parameters (see Proposition \ref{prop:Fcsplitting}). Note that both the holomorphic part and the non-holomorphic part arise from the \emph{same} cusp form $f_{1,32,D,D_0}$. This is the analogue of \eqref{HMFDecomp} above for harmonic Maass forms. One can think of this decomposition as trading exponential growth at the cusps for discontinuities on the upper half-plane given by the local polynomial.
Since the left-hand side is always modular, and the integrals in the first two terms on the right-hand side vanish if and only if $f_{1,32,D,D_0}$ does, \eqref{Lvanishfvanish} and \eqref{Fdecomp3pieces} allow one to conclude that 
\[
P_{1,32,D,D_0} \text{ is modular } \iff L(f\otimes\chi_D,1)=0.
\]

Ordinary polynomials cannot be modular (apart from constant functions in weight $0$). However, there do exist local polynomials which are modular on a congruence subgroup, but this modularity condition is very restrictive (see \cite{BK} for a general discussion and classifications of such objects). Picking $D_0=-3$, these ideas, along with technical arguments extending Kohnen's work to level 32 and dealing with required analytic continuations in this work, gave a finite test for congruent numbers.

For comparison and context, we first review Tunnell's original formula. For odd $n$ (there is a similar formula for even $n$, which we omit), set
\[A_n\coloneqq\#\{(x,y,z)\in\Z^3 \colon n=2x^2+y^2+32z^2\}, \quad B_n\coloneqq\#\{x,y,z,\in\Z^3 \colon n=2nx^2+y^2+8z^2\}.\]
Then $L(f\otimes \chi_n,1)=0$ (and assuming BSD, $n$ is congruent) if and only if $2A_n=B_n$. While this formulas is efficient numerically, it is unlike other formulas in analytic number theory.

The method of locally harmonic Maass forms described above gives a different formula for this result. As Zagier pointed out to the third author, this is a direct analogue of Dirichlet's class number formula. If $d\equiv3\pmod 8$ and $w_{-d}$ is the number of automorphisms of quadratic forms of discriminant $-d$ , then the class number is related to the central $L$-value of $\chi_d\coloneqq\left(\frac{d}{\cdot}\right)$ by the special case of Dirichlet's Class Number Formula:
\[
\frac{w_{-d}\sqrt{-d}\cdot L(\chi_{-d},1)}{2\pi}=\sum_{x^2+y^2+z^2=d}1.
\]
The following theorem gives a ``$\operatorname{GL_2}$ analogue'' of Dirichlet's formula, via a similar formula, but instead of summing over all triples $x,y,z$, only a cone is summed over, and the summand ``$1$'' is twisted by the genus character.
\begin{thm}[\cite{egkr}]\label{MainCongNumThm}
If $D<0$ is a discriminant with $\vt{D}\equiv3\pmod 8$ and $3\vt{D}$ not a square, then on BSD we have that $\vt{D}$ is congruent if and only if 
\[
\begin{aligned}
\sum_{\substack{[a,b,c]\in\mathcal Q_{-3D}\\ c>0>a\\ 32\mid a}}\chi_{-3}([a,b,c])
-
\sum_{\substack{[a,b,c]\in\mathcal Q_{-3D}\\ a+3b+9c>0>a\\ 32\mid a}}\chi_{-3}([a,b,c])
=0
.
\end{aligned}
\]
In particular, $L(f,1)$ is an explicit non-zero multiple of the left hand side, where $f$ is the unique normalized cusp form in $S_2(\Gamma_0(32))$.
\end{thm}

Beyond aesthetic reasons, there are properties of these formulas which are compelling. For instance, a famous result of Monsky \cite{Monsky} is that
\[p\equiv3\pmod 8 \text{ is prime } \implies p \text{ is not congruent}.\]
The authors of \cite{egkr} used a parity argument on the formula above to give a new explanation of this. Since the first sum in Theorem~\ref{MainCongNumThm} doesn't depend on $b$ but only on $b^2$ in the discriminant, and as it turns out that the genus character doesn't depend on $b$ at all, the first sum is invariant under the involution $b\mapsto-b$. Hence, the first sum always contains an even number of summands. Thus, it suffices to show that the set of quadratic forms of discriminant $3p$ with $a+3b+9c>0>a$ and $32\mid a$ contains an {\it odd} number of forms. This was confirmed directly by Genz in unpublished work. 

\subsection{Main results}
In the present paper our main purposes are two-fold. Firstly, we place Theorem \ref{MainCongNumThm} into a much more general (and natural) framework. We restrict to square-free levels to avoid technical complications. In particular, we consider spaces of cusp forms with arbitrary weight $2k \geq 4$ and not necessarily one-dimensional. We restrict to weights greater than $2$ simply for convenience. Note that \eqref{Lvanishfvanish} no longer holds in dimensions greater than one, and so we turn to the theory of Hecke operators to isolate a given newform.
	
Since \cite{egkr} dealt very explicitly with the weight $2$ and level $32$ case and constructing $\mathcal F_{0,N,D,D_0}$ explicitly as a theta lift, much of their paper was devoted to technical details, proving convergence and that certain functions agree where required. We avoid much of this technicality. Moreover, the values $x_{N,1}$ and $x_{N,2}$ that the authors of \cite{egkr} used as their test cases are somewhat mysterious. Here we show that in fact one may choose any rationals to determine the vanishing of the central $L$-values at hand. Further, the conditions for the discriminants $D,D_0$ in the present paper appear more naturally than the special case of \cite{egkr}. This leads to a new proof in the case of $\mathrm{dim}(S_{2k}(N)) = 1$ as well.

To state our main theorem, we require the limit of $P_{k,N,D,D_0}$ towards some $x \in \Q$. This extends a function of Zagier \cite{zagier1} to higher levels. To this end, we define
\begin{align} \label{eq:Pmaindef}
\mathscr{P}_{k,N,D,D_0}(x) \coloneqq \sum_{\substack{Q=[a,b,c] \in \mathcal{Q}_{N, DD_0} \\ a < 0 < Q(x,1)}} \chi_{D_{0}}(Q) Q(x,1)^{k-1}, \qquad x \in \Q.
\end{align}
\begin{thm}\label{Thm: intro}
Let $N$ be square-free and $k > 1$ be an integer. Let $f \in S_{2k}^{new}(N)$ be a Hecke eigenform normalized such that $f(z) = q+O(q^2)$. Let $D$ and $D_0$ be fundamental discriminants with $(-1)^kD$, $(-1)^kD_0>0$, and $\left(\frac{D}{\ell}\right)=\left(\frac{D_0}{\ell}\right)=w_{\ell}$ for all primes $\ell\mid N$, where $w_{\ell}$ is the eigenvalue of $f$ under the Atkin-Lehner involution $W_{\ell}$. Moreover, assume that $D$ and $D_0$ are each perfect squares modulo $4N$, and $DD_0$ is not a perfect square. Choose primes $p_1,\ldots, p_m \nmid N$ and numbers $a_1,\ldots,a_m \in \C$ such that the image of
\begin{align*}
\widetilde{\mathbb{T}} \coloneqq \left(T_{p_1}- p_1^{1-2k} a_{p_1}\right)\cdots\left(T_{p_m}-p_m^{1-2k}a_{p_m}\right)
\end{align*}
is non-trivial and contained in the subspace spanned by $f$. Then, we have
\begin{align*}
L(f\otimes\chi_D,k)L(f\otimes\chi_{D_0},k) = 0
\end{align*}
if and only if
\begin{align*}
\widetilde{\mathbb{T}} \big\vert_{2-2k} \mathscr{P}_{k,N,D,D_0}(x)
\end{align*}
is a constant function of $x\in\Q$.
\end{thm}

\begin{rmks}
\
\begin{enumerate}[leftmargin=*]
\item In level $N=1$, Kong \cite{kongthesis}*{Chapter 4} also considered similar Hecke operators in the context of locally harmonic Maass forms and vanishing of $L$-values. However, the results there are not explicit. In the present paper we give the full treatment in square-free levels.
\item Parallel to the review of this paper, Dombrowsky \cites{dombr26thesis} extended Theorem \ref{Thm: intro} to general odd $N$ based on earlier work by Sataka \cites{sataka05, sataka08}, Ueda \cites{ueda88, ueda93, ueda98} and Ueda--Yamana \cite{ueya}. Her main result (Theorem 3.1) is slightly different from ours, which might arise from a different behavior of the polynomials $\mathscr{P}_{k,N,D,D_0}$ in question.
\item The sum in \eqref{eq:Pmaindef} is in fact a finite sum; Zagier showed this explicitly in \cite{zagier1} (see the proof of \cite{jameson}*{Lemma 2.3} for a corrected version). However, the implied bounds on the coefficients of the quadratic form are impractical to use in the examples below (see Section \ref{Sec: examples}). As we will discuss below, Bengoechea \cite{Bengoechea} studied such sums using continued fraction expansions of $x$, and this framework provides a substantially quicker means to compute the sums in \eqref{eq:Pmaindef} than the naive estimates would give.
\item The Hecke-like polynomial above is of degree at most $\mathrm{dim}_{\C}\left(S_{2k}(N)\right)-1$. In particular, if $\mathrm{dim}_{\C}\left(S_{2k}(N)\right)=1$, then the product $L(f\otimes\chi_D,k)L(f\otimes\chi_{D_0},k)$ vanishes if and only if $\mathscr{P}_{k,N,D,D_0}$ is constant.
\item The conditions on the Kronecker symbols for $D$ and $D_0$ are natural as they exclude cases of trivial $L$-value vanishing due to the sign of the functional equation. 
\item As the polynomial $\mathscr{P}_{k,N,D,D_0}$ has degree $2k-2$, in order to show vanishing of the $L$-values at $D$ and $D_0$, it suffices to evaluate at $2k-1$ rational values $x$. 
\end{enumerate}
\end{rmks}

In \cite{zagier1}, Zagier studied untwisted versions (i.e., without a genus character) of \eqref{eq:Pmaindef} in level $1$. He noted for example, that in situations when the corresponding space of cusp forms is empty that this polynomial is a constant function. Our Theorem \ref{Thm: intro} interprets this as a special case whereby the $L$-values in those cases vanish as the forms in question are identically zero. In a private conversation with the second author, Zagier expected that higher level analogs of his work \cite{zagier1} hold for $\mathscr{P}_{k,N,D,D_0}$ as well. We confirm Zagier's expectations now, which will be key in proving Theorem \ref{Thm: intro}. Generalizing Zagier \cite{zagier1}*{(55)}, we offer the following result.
\begin{thm}\label{thm:zagiersplitting}
Let $N \in \N$ be square-free. Let $D_0$ be a fundamental discriminant, and $D$ be a discriminant such that $\operatorname{sgn}(D_0) = \operatorname{sgn}(D) = (-1)^k$ and that $DD_0$ is not a perfect square. Let $k > 1$,  $c_{\infty}(N,DD_0)$ be the constant from \eqref{cinftydefn}, and $c_{f_{k,N,D,D_0}}$ be the Fourier coefficients of $f_{k,N,D,D_0}$ (see \cite{Kohnen}*{Proposition 2} and recalled in \eqref{eq:fkDFourier} below). We define
\begin{align*}
\gamma_{k,N,DD_0} &\coloneqq \frac{(-1)^k 2^{2k-2}}{\pi \binom{2k-2}{k-1}} c_{\infty}(N,DD_0), \qquad \Phi_{k,N,D,D_0}(x) \coloneqq \sum_{n \geq 1} \frac{c_{f_{k,N,D,D_0}}(n)}{n^{2k-1}} \cos(2\pi n x).
\end{align*}
Then, we have the Fourier expansion
\begin{align*}
\Ps_{k,N,D,D_0}(x) = \gamma_{k,N,DD_0} + \frac{(DD_0)^{k-\frac{1}{2}}\Gamma(k)^2}{(2\pi)^{2k}} \Phi_{k,N,D,D_0}(x).
\end{align*}
\end{thm}

\begin{rmks}
\
\begin{enumerate}[leftmargin=*]
\item 
The function $\Phi_{k,N,D,D_0}$ can be written as $\re\left(\Ec_{f_{k,N,D,D_0}}(x)\right)$, where $\Ec_{f_{k,N,D,D_0}}$ denotes the holomorphic Eichler integral of $f_{k,N,D,D_0}$ (see \eqref{eq:Eichlerholo}) extended to $\R$ in the natural way.
\item Note that Theorem \ref{thm:zagiersplitting} holds in any dimension of $S_{2k}(N)$.
\end{enumerate}
\end{rmks}

In view of Theorem \ref{thm:zagiersplitting}, we obtain a more explicit version of Theorem \ref{Thm: intro}.
\begin{cor} \label{cor:maintheoremexplicit}
Assume the notation and hypotheses from Theorems \ref{Thm: intro} and \ref{thm:zagiersplitting}. Define
\begin{align*}
\gamma_{k,N,DD_0}' &\coloneqq \widetilde{\mathbb{T}}\Big\vert_{2-2k}\gamma_{k,N,DD_0} 
= \left(1+p_1^{1-2k}-p_1^{1-2k}a_{p_1}\right)\cdots\left(1+p_m^{1-2k}-p_m^{1-2k}a_{p_m}\right)\gamma_{k,N,DD_0}.
\end{align*}
Then, we have
\begin{align*}
L(f \otimes \chi_D,k) L(f \otimes \chi_{D_0},k) = 0
\end{align*}
if and only if 
\begin{align*}
\widetilde{\mathbb{T}}\Big\vert_{2-2k} \Ps_{k,N,D,D_0}(x) = \gamma_{k,N,DD_0}'.
\end{align*}
\end{cor}

Corollary \ref{cor:maintheoremexplicit} motivates us to inspect the constant $\gamma_{k,N,DD_0}$ closer. It turns out that they admit a representation in terms of so-called generalized Hurwitz class numbers $H(k,1,\ell,N;n)$, which were introduced by Pei and Wang \cite{PeiWang} in $2003$. We recall their explicit definition in Subsection \ref{Sec: CE series}. These numbers are the Fourier coefficients of the generalized Cohen--Eisenstein series associated to every $\ell \mid N$. These Eisenstein series are modular forms of weight $k+\frac{1}{2}$ and level $4N$ in the plus space (excluding $\ell = 1$ if $k=1$), and thus are higher level analogs of Cohen's classical half-integral weight Eisenstein series \cite{coh75}. If $k=1$, these numbers were investigated in recent work by Beckwith and the second author \cites{bemo, bemo2}. If $k >1$, these numbers were examined by the second author in a follow-up paper \cite{mo3}.
\begin{thm} \label{thm:cinftyexplicit}
Let $k>1$ and $N \in  \N$. Suppose that $N$ is odd and square-free. Let $D_0 > 0$ be a fundamental discriminant, and $D > 0$ be a discriminant. Then, we have
\begin{multline*}
\frac{(2k-1)}{(DD_0)^{k-\frac{1}{2}}} c_{\infty}(N,DD_0) = \frac{\binom{2k-2}{k-1}}{(-1)^k} \gamma_{k,N,DD_0} \\
= \frac{2^{2k-2}}{\pi \zeta(2k)}\sum_{\ell \mid N} \mu(\ell) \Big(\prod_{\substack{p \text{ prime} \\ p \mid \ell}} \frac{1}{1-p^{-2k}}\Big) H(1-k,1,\ell,\ell;D)H(1-k,1,\ell,\ell;D_0).
\end{multline*}
\end{thm}

\begin{rmks}
	\
	\begin{enumerate}[leftmargin=*]
		\item In Theorem \ref{Thm: intro}, both $D$ and $D_0$ are always coprime to the level. In this case, one may apply the functional equation of Dirichlet $L$-functions to obtain an expression in terms of $H(k,1,\ell,\ell;D)H(k,1,\ell,\ell;D_0)$ multiplied by some local factors.
		\item The sums over $\ell \mid N$ resemble Kohnen's \cite{Kohnen}*{eq.\ (3), Theorem 2} description of the Shintani lift in terms of $f_{k,\frac{N}{\ell},D,D_0}$ with $\ell \mid N$.
	\end{enumerate}
\end{rmks}

The remainder of the paper is organized as follows. In Section \ref{Sec: Prelims} we summarize some preliminaries required for the rest of the paper, introducing the various objects we require along with their central results from the literature. Section \ref{Sec: local polys} is dedicated to the behavior of the local polynomials and period-like polynomials for $f_{k,N,D,D_0}$. Section \ref{sec:proofzagiersplitting} provides two proofs of Theorem \ref{thm:zagiersplitting}. In Section \ref{Sec: proof of main theorem} we prove our main Theorem \ref{Thm: intro} by separating the cases $\dim_{\C}(S_{2k}(N)) = 1$ and $\dim_{\C}(S_{2k}(N)) > 1$. Section \ref{sec:proofofcinftyexplicit} is devoted to the proof of Theorem \ref{thm:cinftyexplicit}. We provide some examples to Theorem \ref{Thm: intro} in Section \ref{Sec: examples}. We conclude the paper in Section~\ref{FutureWork} by describing several possible questions for future work.

\section*{Acknowledgements}
We would like to thank Kathrin Bringmann, Markus Schwagenscheidt and Don Zagier for many helpful comments on an earlier version of the paper as well as to Michael Griffin for helpful discussions. Moreover, we would like to thank Charlotte Dombrowsky for pointing out several minor typos in an earlier version. In addition, we are very grateful to the anonymous referee for his thourough review of the manuscript, which improved the exposition significantly.

Part of the research of the first author conducted for this paper is supported by the Pacific Institute for the Mathematical Sciences (PIMS). The research and findings may not reflect those of the Institute. Parts of the second author's research was supported by the CRC/TRR 191 ``Symplectic Structures in Geometry, Algebra and Dynamics'', funded by the DFG (project number 281071066). The third author is grateful for support from a grant from the Simons Foundation (853830, LR), support from a Dean’s Faculty Fellowship from Vanderbilt University, and to the Max Planck Institute for Mathematics in Bonn for its hospitality and financial support. The second and third authors are grateful to the San Souci Institute of Mathematics in Ellijay, Georgia, for its hospitality. In addition, the first three authors are grateful for support of: The Shanks Endowment, Vanderbilt University, and NSF Grant: “Conference: International conference on L-functions and automorphic forms” award number: DMS-2349888.

\section*{Notation}

We provide a list of notation of the most prominent objects used throughout the paper.

\begin{itemize}
	\item Throughout we use the variable $z \in \H$ and let $q \coloneqq e^{2 \pi i z}$.

	\item $S_{2k}(N)$ is the space of cusp forms of weight $2k$ and level $N$.
	
	\item $S_{2k}^{new}(N)$ is the newspace of weight $2k$ cusp forms of level $N$.
	
	\item A \textit{newform} of weight $2k$ and level $N$ is a Hecke eigenform in $S_{2k}^{new}(N)$ which we assume is normalized to have its first Fourier coefficient equal to $1$.
	
	\item $\langle \cdot,\cdot\rangle$ is the Petersson inner-product normalized by $\left[\slz\colon \Gamma_{0}(N) \right]$.
	
	\item $L(f,s)$ is the $L$-function associated to $f$.
	
	\item $L(f\otimes\chi,s)$ is the $L$-function associated to $f$ twisted by a Dirichlet character $\chi$.
	
	\item $r_{f}(z)$ is the period-like polynomial of $f \in S_{2k}(N)$ with even part $r_f^+(z)$.
	
	\item We abbreviate the quadratic form $Q(x,y)=ax^2 +bxy +cy^2$ with discriminant $b^2 -4ac$ by $Q$ or $[a,b,c]$.

	\item $\mathcal{Q}_{D}$ denotes the set of all integral binary quadratic forms with discriminant $D$. 

	\item $\Qc_{N,D}$ is the set of forms $[a,b,c] \in \Qc_D$ such that $N\mid a$.

	\item $[a,b,c]_z = \frac{1}{\im(z)}\left(a\vt{z}^2+b\re(z)+c\right)$ encodes the Heegner geodesic associated to $[a,b,c] \in \Qc_{N,DD_0}$ (given by $S_{[a,b,c]} \coloneqq \{z \in \H \colon a\vt{z}^2+b\re(z)+c=0\}$).

	\item $\chi_{D_{0}} (Q)$ is the extended genus character associated to the discriminant $D_0$.
	
	\item For discriminants $D, D_0 \equiv 0, 1 \pmod{4}$ and $DD_0 >0$, define
	\begin{equation*}
		f_{k,N,D,D_0}(z) \coloneqq \sum_{Q \in \mathcal{Q}_{N,DD_0}} \chi_{D_0}(Q) Q(z,1)^{-k}.
	\end{equation*}

\item $	\mathcal{F}_{1-k,N,D,D_0}(z)$ is the locally harmonic Maass form defined by 
\begin{align*}
\hspace*{\leftmargini} \mathcal{F}_{1-k,N,D,D_0}(z) \coloneqq \frac{1}{2}\sum_{Q \in \mathcal{Q}_{N,DD_0}} \chi_{D_{0}}(Q) \operatorname{sgn}(Q_z) Q(z,1)^{k-1} \beta\left(\frac{DD_0\im(z)^2}{\lvert Q(z,1)\rvert^2}; k-\frac{1}{2},\frac{1}{2}\right),
\end{align*}
where $\beta(x;r,s) \coloneqq \int_{0}^{x} t^{r-1}(1-t)^{s-1}\dm t$ refers to the incomplete $\beta$-function.

\item $E_{N,DD_0} = \bigcup_{Q \in \Qc_{N,DD_0}} \{z \in \H \colon Q_z = 0\}$ is the exceptional set of $\mathcal{F}_{1-k,N,D,D_0}$, namely it contains the jumping singularities of $\mathcal{F}_{1-k,N,D,D_0}$.

\item We define the constant
\begin{align*}
\hspace*{\leftmargini} c_{\infty}(N, DD_0) \coloneqq \frac{(DD_0)^{k-\frac{1}{2}}}{(2k-1)}\pi 2^{2-2k}\sum_{\substack{a \geq 1 \\ N \mid a}} \frac{1}{a^k} \sum_{\substack{0 \leq b < 2a \\ b^2 \equiv DD_0 \pmod{4a}}} \chi_{D_{0}}\left[a,b,\frac{b^2-DD_0}{4a}\right].
\end{align*}

\item If $z \in \H \setminus E_{N,DD0}$, the local polynomial of $\Fc_{1-k,N,D,D0}$ is denoted by
\begin{multline*}
\hspace*{\leftmargini} P_{k,N,D,D_0}(z) \coloneqq c_\infty(N, DD_0) \\
+ (-1)^{k-1} \binom{2k-2}{k-1} \pi 2^{2-2k} \sum_{\substack{Q=[a,b,c] \in \mathcal{Q}_{N,DD_0} \\ a < 0 < Q_z}} \chi_{D_{0}}(Q) Q(z,1)^{k-1}.
\end{multline*}

\item The generalization of Zagier's polynomial are
\begin{align*}
\mathscr{P}_{k,N,D,D_0}(x) \coloneqq \sum_{\substack{Q=[a,b,c] \in \mathcal{Q}_{N, DD_0} \\ a < 0 < Q(x,1)}}. \chi_{D_{0}}(Q) Q(x,1)^{k-1}.
\end{align*}

\item Theorem \ref{thm:zagiersplitting} introduces the notation
\begin{align*}
\hspace*{\leftmargini}
\gamma_{k,N,DD_0} &\coloneqq \frac{(-1)^k 2^{2k-2}}{\pi \binom{2k-2}{k-1}} c_{\infty}(N,DD_0), \quad \Phi_{k,N,D,D_0}(x) \coloneqq \sum_{n \geq 1} \frac{c_{f_{k,N,D,D_0}}(n)}{n^{2k-1}} \cos(2\pi n x).
\end{align*}

\item $T_p$ is the $p$-th Hecke operator.

\item The generalized Hurwitz class numbers $H(k,1,\ell,N;n)$ are defined in Section \ref{Sec: CE series}.

\item $\mathcal{E}_f$ is the holomorphic Eichler integral of $f$.

\item $f^{*}$ is the non-holomorphic Eichler integral of $f$.

\item $W_N$ is the Fricke involution.

\item $\T^{new}$ is the Hecke-like operator defined in Subsection \ref{Sec: HL ops} which maps $S_{2k}(N)$ to $S_{2k}^{new}(N)$.

\item $\T_\nu$ is the Hecke-like operator which maps $S_{2k}^{new}(N)$ to the space generated by a single newform $f_\nu$.
\end{itemize}

\section{Preliminaries}\label{Sec: Prelims}

\subsection{Modular forms}
We recall the \emph{Petersson slash operator}
\begin{align} \label{eq:slashop}
\left(f\vert_k\gamma\right)(z) \coloneqq \begin{cases}
(cz+d)^{-k} f(\gamma z) & \text{if} \ k \in \Z, \\
\left(\frac{c}{d}\right)\varepsilon_d^{2k}(c z+d)^{-k} f(\gamma z) & \text{if} \ k \in \frac{1}{2}+\Z,
\end{cases}
\quad \gamma = \left(\begin{matrix} a & b \\ c& d \end{matrix}\right) \in \begin{cases}
\slz & \text{if } k \in \Z, \\
\Gamma_0(4) & k \in \frac{1}{2}+\Z,
\end{cases}
\end{align}
where $\left(\frac{c}{d}\right)$ denotes the extended Legendre symbol, and $\varepsilon_d \coloneqq 1,i$ if $d \equiv \pm1 \pmod{4}$ for odd integers $d$ (this is guaranteed whenever $\gamma \in \Gamma_0(4)$).
It gives rise to (weakly) holomorphic modular forms in the usual way (see \cite{cohenstromberg} for example).

We use Kohnen's normalization of the Petersson inner product throughout (see \cite{Kohnen}). For $f,g$ two cusp forms of weight $\kappa \in \frac{1}{2}\Z$ on some subgroup $\Gamma$ of finite index in $\slz$ we define
\begin{align*}
	\langle f,g \rangle \coloneqq \frac{1}{\left[ \slz \colon \Gamma \right]} \int_{\Gamma \backslash \H} f(z) \overline{g(z)} y^{\kappa} \frac{\dm x\dm y}{y^2}.
\end{align*}

Recall the concept of the newspace and oldspace of modular forms. Given a modular form $f(z) \in M_k(N,\phi)$ then $f(dz) \in M_k(M,\phi')$ for $M=dN$ and $\phi'$ the induced character. We define the oldspace $S_{k}^{old}(M,\phi)$ to be the image of $S_k(N,\phi)$ under such a map for all non-trivial divisors $d$ of $M$. The newspace $S_{k}^{new}(M,\phi)$ is the orthogonal complement of the oldspace taken in $S_k(M,\phi)$ with respect to the Petersson inner product.

We call an element $f(z) = \sum_{n \geq 1} a(n)q^n \in S_k^{new}(N)$ a \textit{newform} if it is an eigenfunction of all Hecke operators $T_p$ with $p \nmid N$ (see \eqref{eqn: Hecke operators}), and we assume throughout that we normalize to have $a(1)=1$ (such a form is called a normalized newform). It is a classical fact that newforms form a basis for the newspace.

\subsection{Integral binary quadratic forms, genus characters, and cycle integrals}
For any $D$, the group $\Gamma_0(N)$ acts on the set
\begin{align*}
\Qc_{N,D} = \{ax^2+bxy+cy^2 \in \mathcal{Q}_D \colon N \mid a\},
\end{align*}
and this action has finitely many orbits whenever $D \neq 0$. One can check that this action preserves the discriminant $D$, and that it is compatible with the action of $\Gamma$ on $\H$ by fractional linear transformations
in the sense that
\begin{align*}
\left(Q \circ \left(\begin{smallmatrix} a & b \\ c & d \end{smallmatrix}\right)\right)(z,1) = (cz+d)^2 Q(\gamma z,1).
\end{align*}

We refer to the exposition of Gross, Kohnen, Zagier \cite{grokoza}*{p.\ 508} to introduce a genus character $\chi_{D_0}$ for level $N\geq 1$ and $D_0 \neq 0$. Moreover, $\chi_{D_0}$ is $\Gamma_0(N)$-invariant and thus the genus character descends to $\mathcal{Q}_{D_0} \slash \Gamma_0(N)$. If $D_0=1$, the character is trivial. We have
\begin{align} \label{eq:genusrules}
\chi_{D_0}(-Q) = \operatorname{sgn}(D_0)\chi_{D_0}(Q), \qquad \chi_{D_{0}}([a,b,c]) = \chi_{D_{0}}([a,-b,c]).
\end{align}
The genus character satisfies a certain multiplicativity property, which can be found in \cite{grokoza}*{Proposition 1}, and as is noted in \cite{grokoza}*{P2 and P4} it is invariant under the action of the Atkin--Lehner involutions; and in particular the Fricke involution.

Central to Kohnen's results given in Proposition \ref{Prop: kohnen} and Proposition \ref{Cor: kohnen} are cycle integrals of modular forms, and so we recall the definitions here. For more background, see e.g., \cites{DIT,duimto10} and the references therein. Let $D > 0$ be a non-square discriminant. 
Let $h$ be a smooth function which transforms like a modular form of weight $2k$. Then, the \emph{weight $k$ cycle integral of $h$} is defined as
\begin{align*}
\int_{\Gamma_Q \backslash S_Q} h(z) Q(z,1)^{k-1} \dm z, \qquad S_{[a,b,c]} \coloneqq \{z \in \H \colon a\vt{z}^2+b\re(z)+c=0\}.
\end{align*}
The orientation of this integral is counterclockwise if $a > 0$, and clockwise if $a < 0$, where $a \neq 0$ is the first coefficient of $Q$. One can verify that the cycle integral depends only on the equivalence class of $Q$ and that it is invariant under modular substitutions. Thus, $\Gamma_Q \backslash S_Q$ projects to a closed circle in a fundamental domain for $\Gamma_0(N)$. For fundamental discriminants $(-1)^kD, (-1)^kD_0>0$ and $f \in S_{2k}(N)$, the twisted trace of cycle integrals of $f$ is given by
\begin{equation*}
r_{k,N,D,D_0}(f) \coloneqq \sum_{Q \in \Gamma_{0}(N) \backslash \mathcal{Q}_{N,DD_0}} \chi_{D_0}(Q) \int_{\Gamma_Q \backslash S_Q} f(z) Q(z,1)^{k-1} \dm z.
\end{equation*}

\subsection{\texorpdfstring{$L$}{L}-functions and \texorpdfstring{$L$}{L}-values}\label{Sec: L-functions prelims}
We collect some standard properties of $L$-functions associated to modular forms and their quadratic twists, both of which are central objects in the present paper. We follow the expositions by Miyake \cite{miyake}*{Section 4.3} and Zagier \cite{zagier1}*{p.\ 1150}. 

Let $k$, $N \in \N$, and
\begin{align*}
f(z) = \sum_{n\geq 0} a_f(n) \mathrm{e}^{2\pi i n z}
\end{align*}
be holomorphic on $\H$, such that 
\begin{enumerate}[leftmargin=*]
\item its Fourier expansion converges absolutely and uniformly on any compact subset of $\H$,
\item there exists some $\varepsilon > 0$ with $f(z) \in O\left(\im(z)^{-\varepsilon}\right)$ as $\im(z) \to 0$ uniformly in $\re(z)$.
\end{enumerate}
We define the $L$-function associated to $f$ as
\begin{align*}
L(f,s) \coloneqq \sum_{n \geq 1} \frac{a_f(n)}{n^s}, \qquad \re(s) > 1+\varepsilon.
\end{align*}

Since modular forms can be twisted by a Dirichlet character $\chi$ (not to be confused with a genus character), we obtain twisted $L$-functions
\begin{align*}
L(f\otimes\chi,s) \coloneqq \sum_{n \geq 1} \frac{a_f(n)\chi(n)}{n^s}.
\end{align*}
One may also consider $L$-functions associated to a Dirichlet character $\chi_D=\left(\frac{D}{\cdot}\right)$, namely
\begin{align*}
L(s,\chi) \coloneqq \sum_{n \geq 1} \frac{\chi(n)}{n^s} = \prod_{p \text{ prime}} \frac{1}{1-\chi(p)p^{-s}}, \qquad \re(s) > 1.
\end{align*}
They satisfy a certain functional equation which continues $L(s,\chi)$ meromorphically to the whole $s$-plane. (If $\chi$ is principal then $L(s,\chi)$ has a simple pole at $s=1$.) We reserve the notation $L_N(s,\chi)$ for the incomplete $L$-function
\begin{align*}
L_N(s,\chi) \coloneqq \sum_{\gcd(n,N) = 1} \frac{\chi(n)}{n^s} = \prod_{\substack{p \text{ prime} \\ p \nmid N}} \frac{1}{1-\chi(p)p^{-s}}, \qquad \re(s) > 1.
\end{align*}

\subsection{Work of Kohnen}
Following work of Kohnen \cite{Kohnen} and Zagier \cites{zagier75}, we recall the function 
\begin{align*}
f_{k,N,D,D_0}(z) = \sum_{Q \in \Qc_{N,DD_0}} \frac{\chi_{D_{0}}(Q)}{Q(z,1)^k}, \qquad k \geq 2, \qquad N \in \N, \qquad DD_0 > 0.
\end{align*}
\begin{prop}[\protect{\cite{Kohnen}*{Proposition 7}}]\label{Prop: kohnen}
Let $f\in S_{2k}(N)$. Then, we have that
\begin{align*}
\langle f,f_{k,N,D,D_0}\rangle = \frac{\pi\binom{2k-2}{k-1}2^{2-2k}\left(\vt{DD_0}\right)^{\frac{1}{2}-k}}{\left[ \mathrm{SL}_2(\Z) \colon \Gamma_0(N) \right]}r_{k,N,D,D_0}(f).
\end{align*}
\end{prop}

\begin{rmk}
	Although Kohnen \cite{Kohnen} worked with $N$ square-free and odd, the assumption that $N$ is odd is neither needed in his Proposition $7$ nor in his Corollary $3$. Moreover, Ueda and Yamana \cite{ueya} extended Kohnen's results on the Shimura / Shintani isomorphism between $S_{2k}(N)$ and $S_{k+\frac{1}{2}}^+(4N)$ to even levels $N$ including their Hecke equivariance.
\end{rmk}

Moreover, we require the following connection.
\begin{prop}[\protect{\cite{Kohnen}*{Corollary 3}}]\label{Cor: kohnen}
Let $f(z) = \sum_{n \geq 1} a(n)q^n \in S_{2k}^{new}(N)$ be a Hecke eigenform normalized such that $a(1) =1$. Let $D$ and $D_0$ be fundamental discriminants with $(-1)^kD,(-1)^kD_0>0$, and $\left(\frac{D}{\ell}\right)=\left(\frac{D_0}{\ell}\right)=w_{\ell}$ for all primes $\ell\mid N$, where $w_{\ell}$ is the eigenvalue of $f$ under the Atkin-Lehner involution $W_{\ell}$. Then
\begin{align*}
\left(DD_0\right)^{k-1/2}L(f\otimes\chi_D,k)L(f\otimes\chi_{D_0},k)
= \frac{(2\pi)^{2k}}{(k-1)!^2}2^{-2\nu(N)} \vt{r_{k,N,D,D_0}(f)}^2,
\end{align*}
where $\nu(N)$ denotes the number of distinct prime divisors of $N$.
\end{prop}

\begin{rmk}
In \cite{egkr} the authors misquoted Kohnen's result for all cusp forms in $S_{2k}(N)$. This does not affect their results as they considered only dimension one spaces $S_{2k}(N)$ so that the space is generated by a single newform.
\end{rmk}

\subsection{Locally harmonic Maass forms}
Let $f(z) = \sum_{n\geq 1} a_f(n)q^n \in S_k(N)$. We define the \emph{holomorphic Eichler integral} of $f$ by
\begin{align} \label{eq:Eichlerholo}
\mathcal{E}_f(z) &\coloneqq \sum_{n \geq 1} \frac{a_f(n)}{n^{k-1}} q^n = - \frac{(2\pi i)^{k-1}}{(k-2)!}\int_{z}^{i\infty} f(w)(z-w)^{k-2} \dm w
\end{align}
and the \textit{non-holomorphic Eichler integral} of $f$ by
\begin{align} \label{eq:Eichlernonholo}
f^{*}(z) &\coloneqq \sum_{n \geq 1} \frac{\overline{a_f(n)}}{(2\pi n)^{k-1}} \Gamma(k-1,4\pi n \im(z)) q^{-n} = (2i)^{1-k}\int_{-\overline{z}}^{i\infty} \overline{f(-\overline{w})}\left(-i(w+z)\right)^{k-2} \dm w,
\end{align}
where $\Gamma(s,z)$ is the usual incomplete gamma function.
 
Furthermore, we recall $\xi_k \coloneqq 2i \im(z)^k \overline{\frac{\partial}{\partial \overline{z}}}$ from \eqref{eq:xidef}, and the \emph{Bol operator} $\mathbb{D}^{k-1}$ with $\mathbb{D} \coloneqq \frac{1}{2\pi i} \frac{\partial}{\partial z}$ from \eqref{eq:boldef}. Then, one can verify straightforwardly that
\begin{equation} \label{eq:surjectivity}
\begin{split}
\begin{aligned}
&\xi_{2-k} \left(f^*(z)\right) = f(z), \qquad &&\mathbb{D}^{k-1} \left(f^*(z)\right) = 0, \\
&\xi_{2-k} \left(\mathcal{E}_{f}(z)\right) = 0, &&\mathbb{D}^{k-1} \left(\mathcal{E}_{f}(z)\right) = f(z).
\end{aligned}
\end{split}
\end{equation}

Let $DD_0 > 0$ be a non-square discriminant. Following work of Bringmann, Kane, and Kohnen \cite{bkk} and an extension by Ehlen, Guerzhoy, Kane and the third author \cite{egkr}, we define
\begin{align*}
\Fc_{1-k,N,D,D_0}(z) \coloneqq \frac{1}{2}\sum_{Q \in \Qc_{N,DD_0}} \chi_{D_{0}}(Q) \operatorname{sgn}\left(Q_z\right)Q(z,1)^{k-1}\beta\left(\frac{DD_0\im(z)^2}{\vt{Q(z,1)}^2};k-\frac{1}{2},\frac{1}{2}\right), 
\end{align*}
for $z \in \H \setminus E_{N,DD_0}$, where $\beta(x;r,s)$ denotes the incomplete $\beta$-function and
\begin{align*}
E_{N,DD_0} &\coloneqq \bigcup_{Q \in \Qc_{N,DD_0}} S_Q, \qquad [a,b,c]_z \coloneqq \frac{1}{\im(z)}\left(a\vt{z}^2+b\re(z)+c\right).
\end{align*}
The function $\Fc_{1-k,N,D,D_0}$ is the archetypal example of a so-called \emph{locally harmonic Maass form}. Such forms are modular of some weight (here $2-2k$), harmonic outside $E_{N,DD_0}$ with respect to the hyperbolic Laplace operator of the same weight, satisfy a certain limit condition on $E_{N,DD_0}$, and are of polynomial growth towards the cusps of $\Gamma_{0}(N)$. A full definition can be found in \cite{bkk}*{Section 2} in the case of level $1$ and in \cite{egkr}*{Section 2} in the case of level $N$.

In addition, their key feature is that they admit a splitting in terms of Eichler integrals of the {\it same} cusp form and a local polynomial. In the case of $\Fc_{1-k,N,D,D_0}$, this local polynomial is explicitly given by
\begin{align}
P_{k,N,D,D_0}(z) &\coloneqq c_\infty(N, DD_0) + (-1)^{k-1} \binom{2k-2}{k-1} \pi 2^{2-2k} \sum_{\substack{Q=[a,b,c] \in \mathcal{Q}_{N,DD_0} \\ a < 0 < Q_z}} \chi_{D_{0}}(Q) Q(z,1)^{k-1}, \label{locPolydefn} \\ 
c_{\infty}(N, DD_0) &\coloneqq \frac{(DD_0)^{k-\frac{1}{2}}}{(2k-1)}\pi 2^{2-2k}\sum_{\substack{a \geq 1 \\ N \mid a}} \frac{1}{a^k} \sum_{\substack{0 \leq b < 2a \\ b^2 \equiv DD_0 \pmod{4a}}} \chi_{D_{0}}\left[a,b,\frac{b^2-DD_0}{4a}\right]. \label{cinftydefn}
\end{align}

Adapting \cite{bkk}*{Theorem 7.1} to our framework straightforwardly yields the following result.
\begin{prop}[\cite{bkk}*{Theorem 7.1}] \label{prop:Fcsplitting}
If $z \in \H \setminus E_{N,DD_0}$ then
\begin{align*}
\mathcal{F}_{1-k,N,D,D_0}(z) = P_{k,N,D,D_0}(z) - (DD_0)^{k-\frac{1}{2}}\frac{(2k-2)!}{(4\pi)^{2k-1}}\mathcal{E}_{f_{k,N,D,D_0}}(z) + (DD_0)^{k-\frac{1}{2}}f_{k,N,D,D_0}^*(z).
\end{align*}
\end{prop}
By \eqref{eq:surjectivity}, this splitting implies that $\mathcal{F}_{1-k,N,D,D_0}$ maps to $f_{k,N,D,D_0}$ under both $\mathbb{D}^{2k-1}$ and $\xi_{2-2k}$. Moreover, a very recent result by Bringmann, the second and third author \cite{bmr} shows that $\mathcal{F}_{1-k,N,D,D_0}$ has eigenvalue $-1$ under the so-called flipping operator.

We further require the classical Hecke operators $T_p$ with $p$ prime such that $p \nmid N$. Recall that the action of $T_p$ on a translation invariant function $h$ is given by
\begin{align}\label{eqn: Hecke operators}
\left(T_p\vert_{2-2k} h \right)(w) = p^{1-2k} h(pw) + p^{-1} \sum_{b \pmod{p}} h\left(\frac{w + b}{p}\right), \qquad w \in \H \cup \R,
\end{align}
see \cite{diashu}*{Proposition 5.2.1} for instance.

\begin{prop} \label{prop:HeckeOnfkND}
Let $p$ be a prime such that $p \nmid N$.
\begin{enumerate}[leftmargin=*, label={\rm (\roman*)}]
\item We have
\begin{align*}
\hspace*{\leftmargini} T_p \vert_{2k} f_{k,N,D,D_0} (z) = f_{k,N,Dp^2,D_0}(z) + p^{k-1}\left(\frac{D}{p}\right)f_{k,N,D,D_0}(z) + p^{2k-1}f_{k,N,\frac{D}{p^2},D_0}(z).
\end{align*}
\item We have
\begin{multline*}
\hspace*{\leftmargini} T_p \vert_{2-2k} \Fc_{1-k,N,D,D_0} (z) \\
= \Fc_{1-k,N,Dp^2,D_0}(z) + p^{-k} \left(\frac{D}{p}\right) \Fc_{1-k,N,D,D_0} (z)+ p^{1-2k} \Fc_{1-k,N,\frac{D}{p^2},D_0} (z).
\end{multline*}
\end{enumerate}
Here, $f_{k,N,\frac{D}{p^2},D_0}$ and $\Fc_{1-k,N,\frac{D}{p^2},D_0}$ are understood to be $0$ if $p^2\nmid D$.
\end{prop}

\begin{proof}
The special case of (i) with $N=1$ and $D_0=1$ goes back to Parson \cite{parson}*{Theorem 4.1}. Adapting her argument for $N=1$ and $D_0=1$ shows (ii), see \cite{bkk}*{Theorem 1.5}. The reader is also pointed to the proof of \cite{zagier2}*{(36)} for an analogous proof of part (i).  The treatment of the genus character can be found in \cite{beng}*{Proposition 2.2} using \cite{zagier2}*{p.\ $292$}. Parson's argument can be generalized to $N > 1$ straightforwardly.
\end{proof}

\subsection{Generalized Cohen--Eisenstein series}\label{Sec: CE series}
Let $k > 1$. Let $N \in \N$ be odd and square-free. We define
\begin{align*}
	\sigma_{\ell,s}(n) &\coloneqq \sum_{\substack{0 < r \mid n \\ \gcd(\ell,r)=1}} r^s, \qquad \chi_t(a) \coloneqq \left(\frac{t}{a}\right),
	\qquad \sigma_{\ell,N,s}(n) \coloneqq \sum_{\substack{0 < r \mid n \\ \gcd(\ell,r)=1 \\ \gcd\left(\frac{n}{r},\frac{N}{\ell}\right)=1}} r^s.
\end{align*}
Let $\ell \mid N$, $\ell \neq N$. Pei and Wang \cite{PeiWang}*{p.\ 103} introduce \emph{generalized Hurwitz class numbers} as
\begin{align*}
	H(k,1,N,N;n) 
	\coloneqq \begin{cases}
		L_N\left(1-2k,\id\right) & \text{if } n=0, \\
		L_N(1-k,\chi_t) \sum\limits_{\substack{a \mid m \\ \gcd(a,N) = 1}} \mu(a) \chi_t(a) a^{k-1} \sigma_{N,2k-1}\left(\frac{m}{a}\right) & \begin{array}{@{}l} \text{if } (-1)^kn=tm^2, \\ t \text{ fundamental}, \end{array}  \\
		0 & \text{else},
	\end{cases}
\end{align*}
as well as
\begin{multline*}
	H(k,1,\ell,N;n) \\
	\coloneqq \begin{cases}
		0 & \text{if } n=0, \\
		L_{\ell}\left(1-k,\chi_{t}\right) \prod\limits_{p\mid \frac{N}{\ell}} \frac{1-\chi_t(p)p^{-k}}{1-p^{-2k}} \sum\limits_{\substack{a \mid m \\ \gcd(a,N) = 1}} \mu(a) \chi_t(a) a^{k-1} \sigma_{\ell,N,1}\left(\frac{m}{a}\right) & \begin{array}{@{}l} \text{if } (-1)^kn=tm^2, \\ t \text{ fundamental}, \end{array} \\
		0 & \text{else}.
	\end{cases}
\end{multline*}

\begin{rmk}
	Here we include the summation condition that $\gcd(a,N)=1$ in the definition of $H(k,1,N,N;n)$. This condition arises naturally during the proof of Theorem \ref{thm:cinftyexplicit} below, see also \cite{bemo}*{p.\ 14} for further justification.
\end{rmk}

These numbers appear as the Fourier coefficients of the \emph{generalized Cohen--Eisenstein series}
\begin{align*}
\sum_{n \geq 0} H(k,1,\ell,N;n) q^n \in E_{k+\frac{1}{2}}^+(4N), \qquad \ell \mid N.
\end{align*}
Here, $E_\kappa^+(N)$ denotes the Eisenstein plus space of weight $\kappa$ and level $N$. The case of $k=1$ was discussed in \cites{bemo, bemo2} recently, the case $k > 1$ can be found in \cite{mo3}.

\section{Local Polynomials and period-like polynomials} \label{Sec: local polys}

\subsection{Local polynomials}
Let $\mathcal{C}_{\mathfrak{a}}$ denote the connected component of $\H \setminus E_{N,DD_0}$ containing the cusp $\mathfrak{a} \in \Q \cup \{i\infty\}$ on its boundary. We recall the local polynomial
\begin{align*}
P_{k,N,D,D_0}(z) = c_\infty(N, DD_0) + (-1)^{k-1} \binom{2k-2}{k-1} \pi 2^{2-2k} \sum_{\substack{Q=[a,b,c] \in \mathcal{Q}_{N,DD_0} \\ a < 0 < Q_z}} \chi_{D_{0}}(Q) Q(z,1)^{k-1}
\end{align*}
from \eqref{locPolydefn}. It captures the singularities of $\Fc_{1-k,N,D,D_0}$ along the geodesics in $E_{N,DD_0}$.

Following \cite{brimo1}*{Section 2}, We say that a function $f\colon\H\setminus E_{N,DD_0}\to\C$ has {\it jumping singularities} on $E_{N,DD_0}$ if there exists $z \in E_{N,DD_0}$ such that
\begin{align*}
\lim_{\varepsilon\to 0^+} (f(z+i\varepsilon)-f(z-i\varepsilon)) \in \C \setminus \{0\}.
\end{align*}
Note that this limit might depend on the geodesic $S_Q$ on which $z$ is located. Hence, the shape of $P_{k,N,D,D_0}$ depends on the connected component of $\H \setminus E_{N,DD_0}$ in which $z$ is located, and we make this more precise now. Since $DD_0$ is not a square, each geodesic $S_Q$ divides $\H$ into a bounded and an unbounded component, and we denote the bounded component (``interior'') of $\H\backslash S_Q$ by $A_Q$. Since the class number $h(DD_0)$ is finite, there is precisely one unbounded connected component in a fundamental domain for $\Gamma_0(N)$ and this component contains the cusp $i\infty$ on its boundary. Hence, we refer to this connected component as $\mathcal{C}_{i\infty}$. If $z \in \H$ satisfies $y > \frac{\sqrt{DD_0}}{2N}$ then $z \in \mathcal{C}_{i\infty}$, because $S_{[a,b,c]}$ is a semicircle centered at $-\frac{b}{2a}$ and with radius $\frac{\sqrt{DD_0}}{2\vt{a}}$ (recall that $N \mid a$). We identify $\mathcal{C}_{i\infty}$ with all the $\Gamma_0(N)$-equivalent points on $\H$ subsequently. 

We consider the characteristic function
\begin{align*}
\mathbbm{1}_Q(z) \coloneqq \begin{cases}
1 & \text{ if } z \in A_Q, \\
0 & \text{ if } z \not\in A_Q,
\end{cases}
\end{align*}
of $A_Q$ whenever $z \in \H \setminus E_{DD_0}$, see \cite{schw18}*{Corollary 5.3.5} too. Moreover, we let
\begin{align*}
\mathrm{sgn}\left([a,b,c]\right) \coloneqq \mathrm{sgn}(a).
\end{align*}
We begin by citing the following lemma.
\begin{lemma}[\cite{mo2}*{Lemma 2.1 (iii)}]
If $z \in \H \setminus E_{DD_0}$ then
\begin{align*}
\mathrm{sgn}\left(Q_{z}\right) = \mathrm{sgn}(Q) \left(1-2\mathbbm{1}_Q(z)\right).
\end{align*}
\end{lemma}
Hence, if $z \in \mathcal{C}_{i\infty}$ then $\mathrm{sgn}\left(Q_{z}\right) = \mathrm{sgn}(Q)$ and hence the local polynomial reduces to
\begin{align*}
P_{k,N,D,D_0}(z) = c_\infty(N, DD_0)
\end{align*}
for such $z$. This generalizes \cite{bkk}*{(7.6)} to higher levels.

A classical fact is that a (globally defined) translation invariant polynomial has to be constant. However, this is no longer true in the case of local polynomials. 
\begin{lemma} \label{lem:locPolytranslation}
The local polynomial $P_{k,N,D,D_0}$ is translation invariant. In other words, we have $P_{k,N,D,D_0}(z+1) = P_{k,N,D,D_0}(z)$ for every $z, z+1 \in \H \setminus E_{N,DD_0}$.
\end{lemma}

\begin{rmk}
Although the values at $z$ and $z+1$ agree, the set of quadratic forms contributing to the sum defining $P_{k,N,D,D_0}$ might change (if $z$ and $z+1$ are not in the same connected component).
\end{rmk}

\begin{proof}
This follows by Proposition \ref{prop:Fcsplitting} along with the Fourier expansions of both Eichler integrals and modularity of $\mathcal{F}_{1-k,N,D,D_0}$. An alternative computational proof is provided in \cite{mo2}*{Lemma 2.1 (i)} recalling that $\chi_{D_{0}}$ is $\Gamma_{0}(N)$-invariant.
\end{proof}

We specialize the local polynomial $P_{k,N,D,D_0}$ from \eqref{locPolydefn} to connected components $\mathcal{C}_{\mathfrak{a}}$ now. 

\begin{prop} \label{prop:locPolylimit}
Let $x \in \Q$, $\mathscr{P}_{k,N,D,D_0}(x)$ be as in \eqref{eq:Pmaindef}, and $DD_0>0$ be a non-square discriminant. Then, we have
\begin{align*}
\lim_{z \to x} \Big(\sum_{\substack{Q=[a,b,c] \in \mathcal{Q}_{N,DD_0} \\ a < 0 < Q_z}} \chi_{D_{0}}(Q) Q(z,1)^{k-1} \Big)
 = \sum_{\substack{Q=[a,b,c] \in \mathcal{Q}_{N,DD_0} \\ a < 0 < Q(x,1)}} \chi_{D_{0}}(Q) Q(x,1)^{k-1},
\end{align*}
where the limit is understood within a connected component containing $x$ on its boundary. In other words, we have
\begin{align*}
\lim_{z \to x} P_{k,N,D,D_0}(z) = c_\infty(N, DD_0) + (-1)^{k-1} \binom{2k-2}{k-1} \pi 2^{2-2k} \mathscr{P}_{k,N,D,D_0}(x).
\end{align*}
\end{prop}

\begin{proof}
Since $z \in \H$, the summation condition $[a,b,c]_z > 0$ inside $P_{k,N,D,D_0}$ is equivalent to $\im(z)Q_{z} = a\vt{z}^2+b\re(z)+c>0$. For any $x \in \Q$, we observe that 
\begin{align*}
\lim_{z \to x} (\im(z)Q_{z}) = \lim_{z \to x} \left(a\vt{z}^2+b\re(z)+c\right) = Q(x,1) \neq 0,
\end{align*}
as the zeros of $Q(z,1)$ are real quadratic irrationals for non-square discriminants.
\end{proof}

Choosing $x=0$ in the proof of Proposition \ref{prop:locPolylimit} shows that \cite{bkk}*{Corollary 7.2} generalizes verbatim to higher levels.
\begin{cor} \label{cor:BKKCor7.2}
For every $z\in \mathcal{C}_0$, we have
\begin{align*}
P_{k,N,D,D_0}(z) = c_\infty(N, DD_0) + (-1)^{k-1} \binom{2k-2}{k-1} \pi 2^{2-2k} \sum_{\substack{Q=[a,b,c] \in \mathcal{Q}_{N,DD_0} \\ a < 0 < c}} \chi_{D_{0}}(Q) Q(z,1)^{k-1}.
\end{align*}
\end{cor}

Lastly, we show that $\Fc_{1-k,N,D,D0}$ is modular for the Fricke involution.
\begin{prop} \label{prop:FcalFricke}
We have
\begin{align*}
\Fc_{1-k,N,D,D_0}\Big\vert_{2-2k}(W_N-\id) = 0.
\end{align*}
\end{prop}

\begin{proof}
Let $Q=[a,b,c] \in \Qc_{DD_0}$. The actions of $W_N$ on $Q$ and $z$ are compatible, that is
\begin{align*}
\left([a,b,c] \circ W_N\right)(z,1) &= \frac{a}{N}-bz+cNz^2 
= \left(\sqrt{N}z\right)^2[a,b,c]\left(-\frac{1}{Nz},1\right).
\end{align*}
Moreover,
\begin{multline*}
[a,b,c]_{W_Nz} = \frac{1}{\im\left(-\frac{1}{Nz}\right)}\left(a\left\vert -\frac{1}{Nz}\right\vert^2+b\re\left(-\frac{1}{Nz}\right)+c\right) \\
= \frac{N\vt{z}^2}{\im(z)}\left(\frac{a}{N^2}\frac{1}{\vt{z}^2}-\frac{b}{N\vt{z}^2}\re(z)+c\right) 
= \left[cN,-b,\frac{a}{N},\right]_{z} = \left([a,b,c] \circ W_N\right)_{z},
\end{multline*}
and 
\begin{align*}
\frac{DD_0\im\left(-\frac{1}{Nz}\right)}{\left\vert Q\left(-\frac{1}{Nz},1\right) \right\vert^2} = \frac{\frac{1}{N\vt{z}^2}}{\vt{\left(\sqrt{N}z\right)^{-2}}}\frac{DD_0\im\left(z\right)}{\left\vert \left(Q \circ W_N\right)(z,1) \right\vert^2} = \frac{DD_0\im\left(z\right)}{\left\vert \left(Q \circ W_N\right)(z,1) \right\vert^2}.
\end{align*}
The discriminant of $[a,b,c] \circ W_N = \big[cN,-b,\frac{a}{N},\big]$ equals the discriminant of $[a,b,c]$, and the genus character $\chi_{D_{0}}$ is invariant under $W_N$ as well (see \cite{grokoza}*{Proposition 1}). Hence, the claim follows by substituting $Q \circ W_N \mapsto Q$ inside the series defining $\Fc_{1-k,N,D,D0}$ (recall \eqref{eq:genusrules}).
\end{proof}

\subsection{Period-like polynomials}
The theory of period polynomials for $\Gamma_{0}(N)$ is treated in \cites{diamantis, papo, antoniadis}. While we do not require the full theory for $\Gamma_(N)$ (i.e. vector-valued period polynomials with $[\slz \colon \Gamma_{0}(N)]$ many components), we introduce some key objects that we need in the following. By slight abuse of terminology, we call our polynomials ``period-like'', because they specialize to the classical (even) period polynomials in level $N=1$.

For $f \in S_{2k}(N)$ we define
\begin{equation*}
r_{n}(f) \coloneqq \int_{0}^{\infty} f(it) t^n \dm t = \frac{n!}{(2 \pi)^{n+1}} L(f,n+1), \qquad 0 \leq n \leq 2k-2,
\end{equation*}
where $L(f,n+1)$ denote the special $L$-values as defined in Section \ref{Sec: L-functions prelims}. The numbers $r_{n}(f)$ naturally appear as coefficients of the period-like polynomial
\begin{equation*}
r_{f}(z) \coloneqq \int_{0}^{i \infty} f(w) (z-w)^{2k-2} \dm w =\sum_{n=0}^{2k-2} i^{1-n} \binom{2k-2}{n} r_{n}(f) z^{2k-2-n}
\end{equation*}
of $f$. The even period-like polynomial of $f$ is given by
\begin{equation*}
r_{f}^{+}(z) \coloneqq \sum_{\substack{n=0 \\ n \text{ even}}}^{2k-2} i^{1-n} \binom{2k-2}{n} r_{n}(f) z^{2k-2-n}.
\end{equation*}
In addition, we let $\gamma \in \Gamma_{0}(N)\cup\{W_N\}$ and define the polynomials
\begin{align} \label{eq:obstructionpolys}
R_{\gamma}(z) &\coloneqq \Ec_{f_{k,N,D,D_0}}\Big\vert_{2-2k}\left(\gamma-\id\right)(z), \qquad r_{\gamma}(z) \coloneqq f_{k,N,D,D_0}^*\Big\vert_{2-2k}\left(\gamma-\id\right)(z)
\end{align}
of degree at most $2k-2$, where the Eichler integrals were defined in \eqref{eq:Eichlerholo} and \eqref{eq:Eichlernonholo}.
A standard calculation (see \cite{CKL}*{Proposition 3.2} for example) shows that 
\begin{align} \label{eq:periodpolyEichlerholo}
\Ec_{f}\Big\vert_{2-2k}\left(\gamma-\id\right)(z) = -\frac{(2\pi i)^{2k-1}}{(2k-2)!} \int_{\gamma^{-1}i\infty}^{i \infty} f(w) (z-w)^{2k-2} \dm w
\end{align}
for every $\gamma \in \Gamma_{0}(N)$. If $\gamma = W_N$ then $W_N^{-1}i\infty = 0$ and we recover the period-like polynomial of $f$ up to sign of $f$ under $W_N$.

We adapt \cite{koza84}*{Theorem 4} to higher levels.
\begin{prop} \label{prop:evenperiodpoly}
The even period-like polynomial of $f_{k,N,D,D_0}$ is given by
\begin{multline*}
\frac{(DD_0)^{k-\frac{1}{2}}}{\binom{2k-2}{k-1}\pi}r_{f_{k,N,D,D_0}}^{+}(x) \\
= -2\sum_{\substack{Q=[a,b,c] \in \mathcal{Q}_{N,DD_0} \\ a<0<c}} \chi_{D_{0}}(Q) Q(x,1)^{k-1} 
- \frac{(-1)^k 2^{2k-2}}{\binom{2k-2}{k-1}\pi}c_{\infty}(N, DD_0)\left(N^{k-1}x^{2k-2}-1\right).
\end{multline*}
\end{prop}

\begin{proof}
We adapt the proof of \cite{bkk}*{Theorem 1.4}, but we replace the modular inversion $S$ by $W_N$ and determine the multiple of $N^{k-1}x^{2k-2}-1$. By Propositions \ref{prop:Fcsplitting} and \ref{prop:FcalFricke}, we have
\begin{multline} \label{eq:Fcmodularsplitting}
0 = \Fc_{1-k,N,D,D0} \big\vert_{2-2k} (W_N - \id) (z) \\
= P_{k,N,D,D_0} \big\vert_{2-2k} (W_N - \id) (z) - (DD_0)^{k-\frac{1}{2}}\frac{(2k-2)!}{(4\pi)^{2k-1}}R_{W_N}(z) + (DD_0)^{k-\frac{1}{2}}r_{W_N}(z),
\end{multline}
where $R_{W_N}$ and $r_{W_N}$ are as in \eqref{eq:obstructionpolys}. Furthermore, the polynomials $R_{W_N}$ and $r_{W_N}$ are connected to each other via mock modular Eichler--Shimura theory. This goes back to Knopp \cite{knopp62}*{(3.22)}. An extension to our non-holomorphic framework and more general groups has been worked out in \cite{CKL}. Using that $P_{k,N,D,D_0}$ is annihilated by both $\xi_{2-2k}$ and $\mathbb{D}^{2k-1}$, \cite{CKL}*{Theorem 1.3} asserts that\footnote{Note that their constant $c_f$ equals $0$ in our case, because $f = f_{k,N,D,D_0}$ is cuspidal. }
\begin{align*}
r_{W_N}(z) = \frac{(2k-2)!}{(-4\pi)^{2k-1}}R_{W_N}^{c}(\overline{z})
\end{align*}
where the superscript denotes complex conjugated coefficients. Hence, \eqref{eq:Fcmodularsplitting} becomes
\begin{align*}
0 =  P_{k,N,D,D_0}\Big\vert_{2-2k}\left(W_N-\id\right)(z) - (DD_0)^{k-\frac12} \frac{(2k-2)!}{(4\pi)^{2k-1}} \left(R_{W_N}(z) + R_{W_N}^{c}(\overline{z})\right).
\end{align*}
Moreover, \eqref{eq:periodpolyEichlerholo} with $f = f_{k,N,D,D_0}$ yields
\begin{align*}
R_{W_N}(z) = \Ec_{f_{k,N,D,D_0}}\Big\vert_{2-2k}\left(W_N-\id\right)(z) = -\frac{(2\pi i)^{2k-1}}{(2k-2)!}r_{f_{k,N,D,D_0}}(z)
\end{align*}
noting that the sign of the eigenvalue of $f_{k,N,D,D_0}$ under $W_N$ is $+1$ (proof as in Proposition \ref{prop:FcalFricke}). We obtain
\begin{align} \label{eq:evenperiodpolytemp}
P_{k,N,D,D_0}\Big\vert_{2-2k}\left(W_N-\id\right)(z) = -i^{2k-1}(DD_0)^{k-\frac12} 2^{1-2k} \left(r_{f_{k,N,D,D_0}}(z) + r_{f_{k,N,D,D_0}}^c(\overline{z})\right)
\end{align}
Note that $f_{k,N,D,D_0}(it)$ is real, because mapping $b \mapsto -b$ yields (recall \eqref{eq:genusrules})
\begin{align} \label{eq:fkDreal}
f_{k,N,D,D_0}(it) 
= \sum_{[a,b,c] \in \Qc_{N,DD_0}} \frac{\chi_{D_0}[a,b,c]}{(-at^2-ibt+c)^{k}} = \overline{f_{k,N,D,D_0}(it)},
\end{align}
Thus, the coefficients $r_n(f_{k,N,D,D_0})$ are real too. This shows that
\begin{align} \label{eq:periodpolycombining}
r_{f_{k,N,D,D_0}}^c(x) + r_{f_{k,N,D,D_0}}(x) = 2ir_{f_{k,N,D,D_0}}^+(x).
\end{align}
By Corollary \ref{cor:BKKCor7.2}, we have
\begin{multline*}
P_{k,N,D,D_0}\Big\vert_{2-2k}\left(W_N-\id\right)(z) = c_\infty(N, DD_0)\left(N^{k-1}z^{2k-2}-1\right) \\
+ (-1)^{k-1} \binom{2k-2}{k-1} \pi 2^{2-2k} \sum_{\substack{Q=[a,b,c] \in \mathcal{Q}_{N,DD_0} \\ a < 0 < c}} \chi_{D_{0}}(Q) Q(z,1)^{k-1} \Big\vert_{2-2k}\left(W_N-\id\right).
\end{multline*}
Using $[a,b,c] \circ W_N = \big[cN,-b,\frac{a}{N}\big]$ and mapping $Q \mapsto -Q$ (giving a factor $(-1)^{k-1}$ and a factor $\mathrm{sgn}(D_0) = (-1)^k$ by \eqref{eq:genusrules}) establishes that
\begin{multline*}
P_{k,N,D,D_0}\Big\vert_{2-2k}\left(W_N-\id\right)(z) = c_\infty(N, DD_0)\left(N^{k-1}z^{2k-2}-1\right) \\
+ (-1)^k\binom{2k-2}{k-1} \pi 2^{3-2k} \sum_{\substack{Q=[a,b,c] \in \mathcal{Q}_{N,DD_0} \\ a < 0 < c}} \chi_{D_{0}}(Q) Q(z,1)^{k-1}.
\end{multline*}
Letting $z \to x \in \R$ within $\mathcal{C}_0$ in \eqref{eq:evenperiodpolytemp} and combining yields the claim.
\end{proof}

\section{Two proofs of Theorem \ref{thm:zagiersplitting}} \label{sec:proofzagiersplitting}

\subsection{First proof of Theorem \ref{thm:zagiersplitting}}
The idea of this proof is sketched briefly in level $N=1$ by Zagier \cite{zagier1}*{(p.\ 1175)}. 
\begin{proof}[First proof of Theorem \ref{thm:zagiersplitting}]
We begin by showing that 
\begin{align*}
\Ps_{k,N,D,D_0}(x+1) = \Ps_{k,N,D,D_0}(x), \qquad \Ps_{k,N,D,D_0}(-x) = \Ps_{k,N,D,D_0}(x)
\end{align*}
for every $x \in \Q$. To this end, we use Proposition \ref{prop:locPolylimit} to write $\Ps_{k,N,D,D_0}(x)$ as a limit of $P_{k,N,D,D_0}(z)$. Then, the first assertion follows by Lemma \ref{lem:locPolytranslation} directly. For the same reason, we may assume that $z \in \mathcal{C}_0$. Then, the second assertion follows by observing that the sum in Corollary \ref{cor:BKKCor7.2} is invariant under the involution $b \mapsto -b$ on $\Qc_{N, DD_0}$ (which captures $x \mapsto -x)$.

Hence, we obtain the Fourier series representation
\begin{align*}
\Ps_{k,N,D,D_0}(x) = c_0 + 2\re\left(\sum_{n \geq 1} c_n e^{2\pi in x}\right), \qquad c_n = \int_{0}^{1} \Ps_{k,N,D,D_0}(x) e^{-2\pi in x} \dm x.
\end{align*}
We compute the Fourier coefficients $c_n$ for every $n \in \N_0$. If $n \geq 1$, we relate them to the Fourier coefficients of $f_{k,N,D, D_0}$ (which turn out to be real too). Following \cite{zagier1}*{Section $8$}, we rewrite
\begin{multline*}
\int_{0}^{1} \Ps_{k,N,D,D_0}(x) e^{-2\pi in x} \dm x \\
= \sum_{\substack{Q=[a,b,c] \in \Qc_{N,DD_0} \slash \Gamma_{\infty} \\ a < 0}} \chi_{D_{0}}(Q) \int_{-\infty}^{\infty} \left(\max\left(0,Q(x,1)\right)\right)^{k-1} e^{-2\pi in x} \dm x.
\end{multline*}
Substituting $x = \frac{-b+t\sqrt{DD_0}}{2a}$ yields
\begin{align*}
Q(x,1) = \frac{DD_0}{4a}(1-t^2),
\end{align*}
and hence
\begin{align*}
\int_{-\infty}^{\infty} \left(\max\left(0,Q(x,1)\right)\right)^{k-1} e^{-2\pi in x} \dm x = \frac{(DD_0)^{k-\frac{1}{2}}}{2^{2k-1}\vt{a}^k} e^{2\pi i n \frac{b}{2a}} \int_{-1}^{1} \left(1-t^2\right)^{k-1} e^{-2\pi i n \frac{t \sqrt{DD_0}}{2a}} \dm t.
\end{align*}
Identifying the cosets (see \cite{ilt22}*{Lemma 3.2}) gives
\begin{multline*}
\sum_{\substack{Q=[a,b,c] \in \Qc_{N,DD_0} \slash \Gamma_{\infty} \\ a < 0}} \chi_{D_0} \left(\left[a,b,\frac{b^2-DD_0}{4a}\right]\right) e^{2\pi i n \frac{b}{2a}} \\
= 
 \sum_{\substack{a \geq 1 \\ N \mid a}}  \sum_{\substack{0 \leq b \leq 2a-1 \\ b^2 \equiv DD_0 \pmod{4a}}} \chi_{D_0} \left(\left[-a,b,-\frac{b^2-DD_0}{4a}\right]\right) e^{2\pi i n \frac{b}{-2a}} \\
= (-1)^k \sum_{\substack{a \geq 1 \\ N \mid a}}  \sum_{\substack{0 \leq b \leq 2a-1 \\ b^2 \equiv DD_0 \pmod{4a}}} \chi_{D_0} \left(\left[a,b,\frac{b^2-DD_0}{4a}\right]\right) e^{2\pi i n \frac{b}{2a}}.
\end{multline*}
If $n = 0$ then \cite{table}*{8.380.1, 8.384.1} yields
\begin{align*}
\int_{-1}^{1} \left(1-t^2\right)^{k-1} \dm t = \frac{\Gamma(k)\sqrt{\pi}}{\Gamma\left(k+\frac{1}{2}\right)} = \frac{2^{2k-2}}{\binom{2k-2}{k-1}\left(k-\frac{1}{2}\right)},
\end{align*}
and we infer
\begin{align} \label{eq:polyaverage}
c_0 = \int_{0}^{1} \Ps_{k,N,D,D_0}(x) \dm x = \frac{(-1)^k 2^{2k-2}}{\pi \binom{2k-2}{k-1}} c_{\infty}(N,DD_0).
\end{align}
If $n \geq 1$, we recall that we are interested in $\re(c_n)$. Then, \cite{table}*{8.411.8} gives 
\begin{align*}
\int_{-1}^{1} \left(1-t^2\right)^{k-1} \cos\left(\frac{2\pi n \sqrt{DD_0}}{2a} t\right) \dm t 
&= \frac{\Gamma(k)\sqrt{\pi}}{\left(\frac{\pi n \sqrt{DD_0}}{2a}\right)^{k-\frac{1}{2}}} J_{k-\frac{1}{2}}\left(\frac{\pi n \sqrt{DD_0}}{a}\right),
\end{align*}
where $J_{\nu}(x)$ denotes the usual $J$-Bessel function. Thus, we obtain 
\begin{align*}
c_n = (-1)^k\sqrt{DD_0}^{k-\frac{1}{2}} \frac{\Gamma(k)}{2^{k-\frac{1}{2}}\pi^{k-1} n^{k-\frac{1}{2}}} \sum_{\substack{a \geq 1 \\ N \mid a}}  \sum_{\substack{0 \leq b \leq 2a-1 \\ b^2 \equiv DD_0 \pmod{4a}}} \frac{\chi_{D_{0}}(Q)}{a^{\frac{1}{2}}}  e^{2\pi i n \frac{b}{2a}}
J_{k-\frac{1}{2}}\left(\frac{\pi n \sqrt{DD_0}}{a}\right).
\end{align*}
Kohnen \cite{Kohnen}*{Proposition 2} computed the Fourier expansion
\begin{align*}
f_{k,N,D,D_0}(z) = \sum_{n \geq 1} c_{f_{k,N,D,D_0}}(n)q^n,
\end{align*}
for square-free $N$. Since $DD_0$ is not a perfect square, we have
\begin{align} \label{eq:fkDFourier}
c_{f_{k,N,D,D_0}}(n) = \frac{(-1)^k\sqrt{2}(2\pi)^{k+1}}{\Gamma(k)\sqrt{DD_0}^{k-\frac{1}{2}}}n^{k-\frac{1}{2}} \sum_{\substack{a \geq 1 \\ N \mid a}}  \sum_{\substack{0 \leq b \leq 2a-1 \\ b^2 \equiv DD_0 \pmod{4a}}} \frac{\chi_{D_{0}}(Q)}{a^{\frac{1}{2}}}  e^{2\pi i n \frac{b}{2a}}
J_{k-\frac{1}{2}}\left(\frac{\pi n \sqrt{DD_0}}{a}\right).
\end{align}
In particular, both $c_n$, $c_{f_{k,N,D,D_0}}(n) \in \R$ for every $n \geq 1$. Combining, we deduce
\begin{align*}
c_n = \frac{(DD_0)^{k-\frac{1}{2}}\Gamma(k)^2}{2^{2k+1}\pi^{2k}} \frac{1}{n^{2k-1}} c_{f_{k,N,D,D_0}}(n),
\end{align*}
and we arrive at
\begin{align*}
\Ps_{k,N,D,D_0}(x) = \frac{(-1)^k 2^{2k-2}}{\pi \binom{2k-2}{k-1}} c_{\infty}(N,DD_0) 
+ \frac{(DD_0)^{k-\frac{1}{2}}\Gamma(k)^2}{(2\pi)^{2k}}\sum_{n \geq 1} \frac{c_{f_{k,N,D,D_0}}(n)}{n^{2k-1}}  \cos(2\pi n x).
\end{align*}
This is the claim.
\end{proof}

\subsection{Second proof of Theorem \ref{thm:zagiersplitting}}

The idea of this proof is sketched in level $N=1$ by Zagier \cite{zagier1}*{(p.\ 1174)}. 
\begin{proof}[Second proof of Theorem \ref{thm:zagiersplitting}]
Generalizing \cite{zagier1}*{eq.\ (26)}, we have
\begin{multline} \label{eq:Zagier(26)}
\Bigg(\sum_{\substack{Q=[a,b,c] \in \mathcal{Q}_{N,DD_0} \\ a < 0 < Q(x,1)}} \chi_{D_{0}}(Q) Q(x,1)^{k-1}\Bigg) \Bigg\vert_{2-2k} \left(W_N-\id\right)(x) \\
= (-1)^k\sum_{\substack{Q=[a,b,c] \in \mathcal{Q}_{N,DD_0} \\ a < 0 < c}} \chi_{D_{0}}(Q) Q(x,1)^{k-1}.
\end{multline}
The details can be adapted from the proof of \cite{zagier1}*{(18)} directly. As in \eqref{eq:fkDreal}, we observe that
\begin{align*}
\overline{f_{k,N,D,D_0}(-\overline{z})} = f_{k,N,D,D_0}(z),
\end{align*}
and hence $f_{k,N,D,D_0}$ has real Fourier coefficients (compare with \eqref{eq:fkDFourier} too). Combining Proposition \ref{prop:evenperiodpoly} with \eqref{eq:Zagier(26)} shows that
\begin{align*}
\Phi_{k,N,D,D_0}(x) = (-1)^k\re\left(\Ec_{f_{k,N,D,D_0}}(x)\right).
\end{align*}
Following \cite{zagier1}*{(55)}, we obtain
\begin{align*}
\Ps_{k,N,D,D_0}(x) = \sum_{\substack{Q=[a,b,c] \in \Qc_{N, DD_0} \\ a < 0 < Q(x,1)}} \chi_{D_{0}}(Q) Q(x,1)^{k-1} = \alpha_{k,N,DD_0} + (-1)^k\beta_{k,N,DD_0}\Phi_{k,N,D,D_0}(x)
\end{align*}
for some constants $\alpha_{k,N,DD_0}$, $\beta_{k,N,DD_0}$. Since $f_{k,N,D,D_0}$ is a cusp form, the constant Fourier coefficient of $\Ec_{f_{k,N,D,D_0}}$ vanishes. In other words, we have
\begin{align*}
\int_{0}^{1}\Phi_{k,N,D,D_0}(x) \dm x = 0,
\end{align*}
and combining with \eqref{eq:polyaverage} shows that (see \cite{zagier1}*{Section 8} too)
\begin{align*}
\gamma_{k,N,DD_0} = \int_{0}^{1} \Ps_{k,N,D,D_0}(x) \dm x = \alpha_{k,N,DD_0}.
\end{align*}
The constant $\beta_{k,N,DD_0}$ can be computed from \eqref{eq:periodpolyEichlerholo}, \eqref{eq:periodpolycombining}, and Proposition \ref{prop:evenperiodpoly}, namely
\begin{align*}
(-1)^k\beta_{k,N,DD_0} = (-1)^k\frac{(2k-2)!}{2i (2\pi i)^{2k-1}} \frac{(DD_0)^{k-\frac{1}{2}}}{\binom{2k-2}{k-1}\pi} = \frac{(DD_0)^{k-\frac{1}{2}}\Gamma(k)^2}{(2\pi)^{2k}},
\end{align*}
in accordance with the first proof.
\end{proof}

\section{Proof of Theorem \ref{Thm: intro} and Corollary \ref{cor:maintheoremexplicit}}\label{Sec: proof of main theorem}

Let 
$
\mathfrak{n} \coloneqq \mathrm{dim}_{\C}(S_{2k}(N)) \geq 1.
$

\subsection{Proof of Theorem \ref{Thm: intro} if \texorpdfstring{$\mathfrak{n}=1$}{n=1}}

Our proof of the one-dimensional case differs from \cite{egkr} due to our setting of higher weights and square-free levels.
\begin{proof}[Proof of Theorem \ref{Thm: intro} if $\mathfrak{n}=1$]
Let $f \in S_{2k}^{new}(N) \setminus \{0\}$ be as in the theorem. Then, 
\begin{align*}
\Ec_f \big\vert_{2-2k} (\gamma - \id) (z), \qquad f^{*}\big\vert_{2-2k} (\gamma - \id) (z)
\end{align*}
are polynomials of degree at most $2k-2$ for every $\gamma \in \Gamma_0(N)$. By the Fourier expansions of $\mathcal{E}_{f}$ and $f^*$ from \eqref{eq:Eichlerholo} and \eqref{eq:Eichlernonholo}, we have 
\begin{align*}
\mathcal{E}_{f}=0 \iff f = 0 \iff f^* = 0.
\end{align*}
Hence, both Eichler integrals of $f$ do not vanish identically. Since $k > 1$, both Eichler integrals are not constant and thus are linearly independent (one is holomorphic and one is anti-holomorphic). Since $k > 1$ and the Eichler integrals are not trivial, both polynomials are never modular. In other words, there exists some $\widehat{\gamma} \in \Gamma_0(N)$ such that
\begin{align*}
\left(- (DD_0)^{k-\frac{1}{2}}\frac{(2k-2)!}{(4\pi)^{2k-1}}\Ec_f + (DD_0)^{k-\frac{1}{2}}f^*\right)\big\vert_{2-2k} (\widehat{\gamma} - \id) (z) \neq 0.
\end{align*}
Here, the constants are chosen as in Proposition \ref{prop:Fcsplitting}.

Since $\Fc_{1-k,N,D,D0}$ is modular and admits the splitting from Proposition \ref{prop:Fcsplitting}, we have
\begin{multline*}
0 = \Fc_{1-k,N,D,D0} \big\vert_{2-2k} (\widehat{\gamma} - \id) (z) = P_{k,N,D,D_0} \big\vert_{2-2k} (\widehat{\gamma} - \id) (z) \\
+ \left(- (DD_0)^{k-\frac{1}{2}}\frac{(2k-2)!}{(4\pi)^{2k-1}}\Ec_{f_{k,N,D,D_0}} + (DD_0)^{k-\frac{1}{2}}f_{k,N,D,D_0}^*\right)\big\vert_{2-2k} (\widehat{\gamma} - \id) (z).
\end{multline*}
By assumption, we have $f_{k,N,D,D_0} = \delta_{k,N,D,D_0} f$ for some $\delta_{k,N,D,D_0} \in \C$. Hence, we obtain
\begin{multline*}
P_{k,N,D,D_0} \big\vert_{2-2k} (\widehat{\gamma} - \id) (z) \\
= \delta_{k,N,D,D_0} \left((DD_0)^{k-\frac{1}{2}}\frac{(2k-2)!}{(4\pi)^{2k-1}}\Ec_{f} - (DD_0)^{k-\frac{1}{2}}f^*\right)\big\vert_{2-2k} (\widehat{\gamma} - \id) (z).
\end{multline*}
We deduce that $\delta_{k,N,D,D_0} = 0$ (and hence $f_{k,N,D,D_0} = 0$) if and only if
\begin{align*}
P_{k,N,D,D_0} \big\vert_{2-2k} (\widehat{\gamma} - \id) (z) = 0
\end{align*}
for every $z \in \H \setminus E_{N,DD_0}$. 

Since $N$ is square-free, we may use Kohnen's results given in Proposition \ref{Prop: kohnen} and Proposition \ref{Cor: kohnen}. Hence, we have that 
\begin{align*}
L(f\otimes\chi_D,k)L(f\otimes\chi_{D_{0}},k) = 0
\end{align*}
if and only if $\langle f, f_{k,N,D,D_0} \rangle = 0$. This is in turn equivalent to $f_{k,N,D,D_0} = 0$ by assumption and the fact that the Petersson inner-product is non-degenerate. Hence, we obtain
\begin{align*}
L(f\otimes\chi_D,k)L(f\otimes\chi_{D_0},k) = 0
\end{align*}
if and only if $f_{k,N,D,D_0} = 0$. In turn, we have $f_{k,N,D,D_0} = 0$ if and only if $\Ec_{f_{k,N,D,D_0}} = 0$ as asserted at the beginning of the proof. By the Fourier expansion of $\mathscr{P}_{k,N,D,D_0}$ in Theorem \ref{thm:zagiersplitting}, we have $\Ec_{f_{k,N,D,D_0}} = 0$ if and only if $\Phi_{k,N,D,D_0} = 0$. In other words, we have $f_{k,N,D,D_0} = 0$ if and only if 
\begin{align*}
\mathscr{P}_{k,N,D,D_0}(x) = \gamma_{k,N,DD_0} = \frac{(-1)^{k} 2^{2k-2}}{\binom{2k-2}{k-1} \pi} c_\infty(N, DD_0)
\end{align*}
for every $x \in \Q$. We deduce Theorem \ref{Thm: intro} in the one-dimensional case. 
\end{proof}

\subsection{Hecke-like operators}\label{Sec: HL ops}
Here we produce Hecke-like operators that will later enable us to isolate a given newform (that is, an eigenform in the newspace) in order for us to return to one-dimensionality arguments in future sections (see also \cite{kongthesis}).

Consider $f \in S_{2k}(N)$. Assume that $\mathfrak{n} \geq 2$ and that $\{ f_1,\dots, f_{\mathfrak{n}} \}$ is a basis of normalized Hecke eigenforms for the space in question. For a prime $p$ we let $a_{p,j}$ be defined by $ T_p \vert_{2k} f_j = a_{p,j} f_j$. Write $f = \sum_{\nu=1}^{\mathfrak{n}} c_\nu f_\nu$. 

Assume that we aim to isolate the eigenform $f_1$. Since $f_1$ and $f_2$ are not equal, there exists a prime $p_2$ such that $a_{p_2,1} \neq a_{p_2,2}$ (that is, there exists a Hecke operator that distinguishes the two eigenforms). Thus we obtain that
\begin{align*}
 (T_{p_2} - a_{p_2,2}) \vert_{2k} (c_1f_1) = ( a_{p_2,1} - a_{p_2,2}) c_1 f_1 \neq 0 \qquad  (T_{p_2} - a_{p_2,2})c_2f_2 = 0.
\end{align*}
Thus we have that
\begin{align*}
 (T_{p_2} - a_{p_2,2})\vert_{2k} f = ( a_{p_2,1} - a_{p_2,2}) c_1 f_1 + \sum_{\nu = 3}^{\mathfrak{n}} ( a_{p_2,\nu} - a_{p_2,2}) c_\nu f_\nu.
\end{align*}
It is clear that one may iterate this process a finite number of times (annihilating each $f_\nu$ in turn) to produce an operator 
\begin{align*}
\T_1 \coloneqq (T_{p_2} -a_{p_2,2})(T_{p_3,3}-a_{p_3,3}) \cdots (T_{p_{\mathfrak{n}}} -a_{p_{\mathfrak{n}},\mathfrak{n}})
\end{align*}
such that
\begin{align*}
 \T_1 \vert_{2k} f  = c_1 (a_{p_2,1} - a_{p_2,2}) (a_{p_3,1} - a_{p_3,3}) \cdots (a_{p_{\mathfrak{n}},1} - a_{p_{\mathfrak{n}},\mathfrak{n}}) f_1 \neq 0.
\end{align*}

Using the same construction, it is clear that one may isolate any given eigenform $f_\nu$ in the basis by constructing the operator $\T_\nu$ in the same fashion.

Thus for us it is clearly enough to study forms $f \in S_{2k}^{new}(N)$. Moreover, for such an $f$ it is clear that one may construct a Hecke-like operator which maps $f$ to the space generated by a given newform (i.e.\ a given eigenspace in $S_{2k}^{new}(N)$). One may then clearly rescale the resulting newform so that $a(1) =1$ in order to apply Proposition \ref{Cor: kohnen}.

For fixed $k,N$ we let 
\begin{align*}
\T^{new} \colon S_{2k}(N) \to S_{2k}^{new}(N)
\end{align*}
be the Hecke-like operator that projects $S_{2k}(N)$ onto the subspace of newforms, and 
\begin{align*}
\T_\nu \colon S_{2k}^{new}(N) \to \mathrm{span}_{\C}\{f_{\nu}\}
\end{align*}
be the Hecke-like operator that projects $S_{2k}^{new}(N)$ to the eigenspace generated by the newform $f_\nu$. By the strong multiplicity one theorem (see \cite{IK}*{Theorem 14.18} for example), the space generated by $f_\nu$ is one-dimensional. The fact that $\T^{new}$ is constructed via Hecke operators ensures ``compatibility'' with the Petersson inner products.

\subsection{Proof of Theorem \ref{Thm: intro} if \texorpdfstring{$\mathfrak{n}>1$}{n>1} and of Corollary \ref{cor:maintheoremexplicit}}

Now, we prove our main result in higher dimensions as well as Corollary \ref{cor:maintheoremexplicit}.
\begin{proof}[Proof of Theorem \ref{Thm: intro} if $n>1$ and of Corollary \ref{cor:maintheoremexplicit}]
Since $f$ is assumed to be a normalized newform we have that $f = f_\nu$ for some basis element $f_\nu$. We construct a Hecke-like operator $\T = \T_\nu \circ \T^{new}$ as defined in Subsection \ref{Sec: HL ops}, where 
\begin{align*}
\T^{new} \colon S_{2k}(N) \rightarrow S_{2k}^{new}(N), \qquad T_\nu \colon S_{2k}^{new}(N) \rightarrow \mathrm{span}_\C\{f_\nu\}.
\end{align*}
This construction is not unique. We recall that the Hecke operators are self-adjoint (because the characters are real). Since $f = f_\nu$ we see that
\begin{align*}
	\langle f, f_{k,N,D,D_0} \rangle = C \langle f, \T f_{k,N,D,D_0} \rangle
\end{align*}
for some constant $C \in \C$. Moreover, the right-hand side vanishes if and only if $\T f_{k,N,D,D_0}$ vanishes since the Petersson inner product is non-degenerate and $\T f_{k,N,D,D_0} \in \mathrm{span}_\C\{f\}$.

The action of a Hecke-like operator $T_p - a_p$ with $p \nmid N$ on $f_{k,N,D,D_0}$ lifts to the action of $T_p - p^{1-2k}a_p$ on the locally harmonic Maass form $\Fc_{1-k,N,D,D_0}$ by combining Propositions \ref{prop:Fcsplitting} and \ref{prop:HeckeOnfkND} with \eqref{eq:surjectivity}. In other words, we have
\begin{align*}
T_p \Big\vert_{2k} \left(\xi_{2-2k} \Fc_{1-k,N,D,D0}\right) &= \xi_{2-2k} \left(T_p \Big\vert_{2-2k} \Fc_{1-k,N,D,D0}  \right), \\
T_p \Big\vert_{2k} \left(\mathbb{D}^{2k-1} \Fc_{1-k,N,D,D0}\right) &= \mathbb{D}^{2k-1} \left(T_p \Big\vert_{2-2k} \Fc_{1-k,N,D,D0}\right).
\end{align*}
Thus we see that the operator $\T$ acting on $f_{k,N,D,D_0}$ also lifts to a Hecke-like operator $\widetilde{\T}$ acting on $\Fc_{1-k,N,D,D_0}$.

The proof of the general case of Theorem \ref{Thm: intro} now follows in the same way as in the one-dimensional case. The operator $\widetilde{\T}$ now acts on $\mathscr{P}_{k,N,D,D_0}$, precisely giving the Hecke-like operator in the theorem statement as well as in Corollary \ref{cor:maintheoremexplicit}.
\end{proof}

\section{Proof of Theorem \ref{thm:cinftyexplicit}} \label{sec:proofofcinftyexplicit}

In this section we compute an explicit formula for the constant $c_\infty(N,DD_0)$ in terms of generalized Hurwitz class numbers. We begin with a preparatory technical lemma.
\begin{lemma} \label{lem:levelreduction}
	Let $N$ be square-free and $a \geq 1$. Then, it holds that
	\begin{align*}
		\sum_{\substack{\ell \mid N \\ \gcd(\ell,a)=1}} \mu(\ell) = \sum_{\ell \mid \frac{N}{\gcd(N,a)}} \mu(\ell) = 
\begin{cases}
			1 & \text{if } N = \gcd(N,a),  \text{i.e. } N \mid a, \\
			0 & \text{otherwise}.
		\end{cases}
	\end{align*}
\end{lemma}


We use this result along with results of Wong \cite{wong}*{eq.\ (2)} pertaining to the case of level $1$.
\begin{proof}[Proof of Theorem \ref{thm:cinftyexplicit}]	
	We first reduce the calculation to the $N=1$ case. For $a \in \N$, define
	\begin{align*}
		\psi(a) \coloneqq \sum_{\substack{0 \leq b < 2a \\ b^2 \equiv DD_0 \pmod{4a}}} \chi_{D_{0}}\left[a,b,\frac{b^2-DD_0}{4a}\right],
	\end{align*}
	which, by virtue of \cite{wong}*{Proposition $1$}, is a multiplicative function of $a$. Lemma \ref{lem:levelreduction} enables us to write
	\begin{align*}
		\sum_{\substack{a \geq 1 \\ N \mid a}} \frac{\psi(a)}{a^k} 
		= \sum_{\ell \mid N} \mu(\ell) \sum_{a \geq 1} \chi_{\ell}(a)^2 \frac{\psi(a)}{a^k} = \sum_{\ell \mid N} \mu(\ell) \prod_{\substack{p \text{ prime} \\ p \nmid \ell}} \sum_{j \geq 0} \frac{\psi\left(p^j\right)}{p^{jk}}.
	\end{align*}
	Let $D = \widetilde{D}f_D^2$ with fundamental discriminant $\widetilde{D} > 0$. The inner sum is evaluated in Wong's paper \cite{wong}*{Proposition $2$}, which states (with $\nu_p$ the $p$-adic valuation) that
	\begin{multline*}
		\sum_{j\geq 0} \frac{\psi\left(p^j\right)}{p^{jk}} = \frac{1-p^{-2k}}{\left(1-\chi_{\widetilde{D}}(p)p^{-k}\right)\left(1-\chi_{D_{0}}(p)p^{-k}\right)} \\
		\times \frac{1}{\left(p^{\nu_p\left(f_D\right)}\right)^{2k-1}} \left(\sigma_{2k-1}\left(p^{\nu_p\left(f_D\right)}\right) - \chi_{\widetilde{D}}(p)p^{k-1}\sigma_{2k-1}\left(p^{\nu_p\left(f_D\right)-1}\right)\right),
	\end{multline*}
	where we adopt the usual convention $\sigma_{2k-1}(1/p) \coloneqq 0$. 
	
	The next step is to rewrite the divisor sums that arise. Since $p \nmid \ell$, we obtain
	\begin{multline*}
		\prod_{\substack{p \text{ prime} \\ p \nmid \ell}} \left(\sigma_{2k-1}\left(p^{\nu_p\left(f_D\right)}\right) - \chi_{\widetilde{D}}(p)p^{k-1}\sigma_{2k-1}\left(p^{\nu_p\left(f_D\right)-1}\right)\right) \\
		= \prod_{\substack{p \text{ prime} \\ p \mid f_D, \ p \nmid \ell}} \left(\sigma_{\ell, 2k-1}\left(p^{\nu_p\left(f_D\right)}\right) - \chi_{\widetilde{D}}(p)p^{k-1}\sigma_{\ell, 2k-1}\left(p^{\nu_p\left(f_D\right)-1}\right)\right).
	\end{multline*}
	Decomposing $f_D$ into prime powers, a standard argument using multiplicativity shows that
	\begin{multline*}
		\prod_{\substack{p \text{ prime} \\ p \mid f_D, \ p \nmid \ell}} \left(\sigma_{\ell,2k-1}\left(p^{\nu_p\left(f_D\right)}\right) - \chi_{\widetilde{D}}(p)p^{k-1}\sigma_{\ell,2k-1}\left(p^{\nu_p\left(f_D\right)-1}\right)\right) \\
		= 
		\sum_{\substack{r \mid f_D \\ \gcd(\ell,r)=1}} \mu(r) \chi_{\widetilde{D}}(r) r^{k-1}\sigma_{\ell,2k-1}\left(\frac{f_D}{r}\right).
	\end{multline*}
	Moreover, we have
		\begin{multline*}
			\Bigg(\prod_{\substack{p \text{ prime} \\ p \mid f_D, \ p \nmid \ell}} \frac{1}{\left(p^{\nu_p\left(f_D\right)}\right)^{2k-1}} \Bigg) \sum_{\substack{r \mid f_D \\ \gcd(\ell,r)=1}} \mu(r) \chi_{\widetilde{D}}(r) r^{k-1}\sigma_{\ell,2k-1}\left(\frac{f_D}{r}\right) \\
			= \sum_{\substack{r \mid f_D \\ \gcd(\ell,r)=1}} \mu(r) \chi_{\widetilde{D}}(r) r^{-k}\sigma_{\ell,1-2k}\left(\frac{f_D}{r}\right),
		\end{multline*}
		which is an adaptation of \cite{PeiWang}*{p.\ 114}. Combining yields
		\begin{multline*}
			\sum_{\substack{a \geq 1 \\ N \mid a}} \frac{1}{a^k} \sum_{\substack{0 \leq b < 2a \\ b^2 \equiv DD_0 \pmod{4a}}} \chi_{D_{0}}\left[a,b,\frac{b^2-DD_0}{4a}\right] \\
			= \sum_{\ell \mid N} \mu(\ell) \frac{L_{\ell}\left(k,\chi_{D_0}\right) L_{\ell}\left(k,\chi_{\widetilde{D}}\right)}{L_{\ell}\left(\id,2k\right)} \sum_{\substack{r \mid f_D \\ \gcd(\ell,r)=1}} \mu(r) \chi_{\widetilde{D}}(r) r^{-k}\sigma_{\ell,1-2k}\left(\frac{f_D}{r}\right),
		\end{multline*}
		which establishes the claim by definition of the numbers $H(1-k,1,\ell,\ell;n)$ (note that if $\kappa<1$ then $H(\kappa,1,N,N;n)$ is still well-defined by the functional equation of Dirichlet $L$-functions.).
\end{proof}

\section{Numerical examples}\label{Sec: examples}

We present some numerical examples to Theorem \ref{Thm: intro}. The underlying computer code can be found in the ancillary files to our arXiv submission. 

The finiteness of $\{Q \in \mathcal Q_{N,DD_0} \colon a<0<Q(x,1)\}$ for rational $x=p/q$ is shown by the identity \[Dq^2=\vt{bq+2ap}^2+4\vt{a}\cdot\vt{ap^2+bpq+cq^2}\] from \cite{zagier1} (note that \cite{jameson} corrected a typo there). The resulting bounds on $a,b$ are still much too large to use in our numerical examples. In order to prove convergence of similar sums to ours for irrational values of $x$, Bengoechea \cite{Bengoechea} provided (two) means via continued fractions to compute the forms in this set which dramatically speeds up convergence. The code employs simple forms hit with matrices depending on the continued fraction of $x$, which is directly based on the bijection in her Theorem~3.1.

\subsubsection{Weight \texorpdfstring{$4$}{4}, Level \texorpdfstring{$7$}{7}}
We note that Rosson and Tornar\'{\i}a \cite{rostor}*{Table 2} investigated this case as well. We have $\mathrm{dim}_{\C}\big(S_4(7)\big) = 1$, and the normalized Hecke newform is given by
\begin{align*}
f_1(\tau) = q - q^2 - 2 q^3 - 7 q^4 + 16 q^5 + 2 q^6 - 7 q^7 + 15 q^8 - 23 q^9 + O\left(q^{10}\right),
\end{align*}
see LMFDB \cite{lmfdb} newform orbit $7$.$4$.a.a. Moreover, it has sign $1$ under the Atkin--Lehner involution $W_7$. In this case, admissible discriminants are squares modulo $28$, non-squares themselves, satisfying $\big(\frac{D}{7}\big) = 1$, being coprime to the level and inducing a primitive Kronecker character. With Pari/GP \cite{pari}, we compute that
\begin{alignat*}{2}
& L\left(f_1 \otimes \left(\frac{29}{\cdot}\right),2\right) = 3.009928487\ldots, \qquad
&& L\left(f_1 \otimes \left(\frac{37}{\cdot}\right),2\right) = 0, \\
& L\left(f_1 \otimes \left(\frac{44}{\cdot}\right),2\right) = 1.610549804\ldots, 
&& L\left(f_1 \otimes \left(\frac{57}{\cdot}\right),2\right) = 0.273074344\ldots, \\
& L\left(f_1 \otimes \left(\frac{92}{\cdot}\right),2\right) = 0.
\end{alignat*}

Next, we show how to compute the values required in the cirterion of Theorem \ref{Thm: intro}. Since we have a one-dimensional space (and in particular no oldspace), there are no Hecke operators involved in this case. An example is given as follows.
\begin{lstlisting}[language=Sage, caption=Sage code for weight 4 and level 7, captionpos=b]
sage: attach('LocalPolyCodeFinal.sage')
sage: k=2
sage: N=7
sage: D0=29
sage: D=37
sage: Plocal(1/2)
144
\end{lstlisting}

The computation at other rational numbers is similar, and we obtain the values listed in Table \ref{table:lvl7table}. 
\begin{table}[htbp]
\centering
\begin{tabular}{| M{1.4cm} | M{2.5cm} M{2.5cm} M{2.5cm} M{2.5cm} |}
\hline
& & & & \\[-8pt]
$D_0=29$ & $D=37$ & $D=44$ & $D=57$ & $D=92$ \\ 
& & & & \\[-8pt]
\hline
& & & & \\[-8pt]
$x=\frac{1}{2}$ & $2^4 \cdot 3^2$ & $2^3\cdot3\cdot7$ & $2^2\cdot3\cdot5\cdot7$ & $2^6\cdot3^2 $ \\
& & & & \\[-8pt]
$x=\frac{1}{9}$ & $2^4 \cdot 3^2$ & $\frac{2^9\cdot5\cdot7}{3^4}$ & $\frac{2^3\cdot7\cdot569}{3^4}$ & $2^6\cdot3^2$ \\
& & & & \\[-8pt]
$x=\frac{1}{3}$ & $2^4 \cdot 3^2$ & $\frac{2^4\cdot5\cdot7}{3}$ & $\frac{2^4\cdot7\cdot11}{3}$ & $2^6\cdot3^2$ \\
& & & & \\[-8pt]
$x=\frac{1}{10}$ & $2^4 \cdot 3^2$ & $\frac{2^2\cdot7\cdot197}{5^2}$ & $\frac{2\cdot7\cdot19\cdot37 }{5^2}$ & $2^6\cdot3^2$ \\
& & & & \\[-8pt]
$x=\frac{1}{5}$ & $2^4 \cdot 3^2$ & $\frac{2^8\cdot3\cdot7}{5^2}$ & $\frac{2^3\cdot3\cdot7\cdot59}{5^2}$ & $2^6\cdot3^2$ \\
& & & & \\[-8pt]
$x=\frac{1}{12}$ & $2^4 \cdot 3^2$ & $\frac{2\cdot7\cdot11\cdot13}{3^2}$ & $\frac{5\cdot7\cdot101}{3^2}$ & $2^6\cdot3^2$ \\
& & & & \\
\hline
\end{tabular}
\caption{Numerical values pertaining to weight $4$ and level $7$.}
\label{table:lvl7table}
\end{table}
In particular, we observe that the values corresponding to some discriminant $D$ coincide for every $x$ listed if and only if $L\big(f_1 \otimes \big(\frac{D}{\cdot}\big),2\big)=0$.

\subsubsection{Weight \texorpdfstring{$4$}{4}, Level \texorpdfstring{$15$}{15}}
We have $\mathrm{dim}_{\C}\big(S_4(15)\big) = 4$, and consider the Hecke newform
\begin{align*}
f_2(\tau) = q + 3 q^2 - 3 q^3 + q^4 - 5 q^5 - 9 q^6 + 20 q^7 - 21 q^8 + 9 q^9 + O\left(q^{10}\right),
\end{align*}
see LMFDB \cite{lmfdb} newform orbit $15$.$4$.a.b. Moreover, it has sign $+1$ under both Atkin--Lehner involutions $W_3$ and $W_5$. In this case, admissible discriminants are squares modulo $60$, non-squares themselves, satisfying $\big(\frac{D}{3}\big) = \big(\frac{D}{5}\big) = 1$, being coprime to the level and inducing a primitive Kronecker character. With Pari/GP \cite{pari}, we compute that
\begin{alignat*}{2}
& L\left(f_2 \otimes \left(\frac{61}{\cdot}\right),2\right) = 0.378936801\ldots, \qquad
&& L\left(f_2 \otimes \left(\frac{76}{\cdot}\right),2\right) = 0.272484089\ldots, \\
& L\left(f_2 \otimes \left(\frac{109}{\cdot}\right),2\right) = 1.42778988\ldots, 
&& L\left(f_2 \otimes \left(\frac{124}{\cdot}\right),2\right) = 0.522984720\ldots,
\end{alignat*}
and that
\begin{align*}
L\left(f_2 \otimes \left(\frac{181}{\cdot}\right),2\right) = L\left(f_2 \otimes \left(\frac{229}{\cdot}\right),2\right) = L\left(f_2 \otimes \left(\frac{1009}{\cdot}\right),2\right) = 0.
\end{align*}

Next, we show how to compute the values required in the cirterion of Theorem \ref{Thm: intro}. As the dimension of $S_4(15)$ is greater than $1$, we require the action of Hecke operators to annihilate oldspaces and newforms being linearly independent from $f_2$. Note that the space $S_4(3)$ is trivial, the space $S_4(5)$ is one-dimensional, and there is precisely one normalized newform $g \in S_4(15)$ being untwisted and linearly independent from $f_2$. We choose the Hecke operator $T_{11}-32$ to annihilate the space $S_4(5)$ and the Hecke operator $T_7+24$ to annihilate $g$, where the primes $7$ and $11$ are coprime to the level. Consequently, our Hecke polynomial is given by
\begin{align*}
\left(T_{11}-11^{-3}\cdot 32\right) \cdot \left(T_7+7^{-3} \cdot 24\right) 
= T_{77}+7^{-3}\cdot24\cdot T_{11}-11^{-3}\cdot32\cdot T_7-768\cdot 77^{-3},
\end{align*}
and the action of this polynomial is encoded as the function ``Heckeaction15'' in our Sage code. Two examples are given as follows.
\begin{lstlisting}[language=Sage, caption=Sage code for weight 4 and level 15, captionpos=b]
sage: attach('LocalPolyCodeFinal.sage')
sage: k=2
sage: N=15
sage: D0=61
sage: D=181
sage: Heckeaction15(1/2)
100684800/41503
sage: D=1009
sage: Heckeaction15(1/2)
2236262400/41503
\end{lstlisting}

The computation at other rational numbers is similar, and we obtain the values listed in Table \ref{table:lvl15table}.
\begin{table}[htbp]
\centering
\begin{tabular}{| M{1.5cm} | M{2.5cm} M{2.5cm} M{2.5cm} |}
\hline
& & & \\[-8pt]
$D_0=61$ & $D=76$ & $D=109$ & $D=124$ \\ 
& & & \\[-8pt]
\hline
& & & \\[-8pt]
$x=\frac{1}{2}$ & $\frac{2^8\cdot3^2\cdot2081}{7^2\cdot11^2}$ & $\frac{2^8\cdot3^2\cdot127\cdot163}{7^3\cdot11^2}$ & $\frac{2^9\cdot3^3\cdot19\cdot269}{7^3\cdot11^2 }$ \\
& & & \\[-8pt]
$x=\frac{1}{17}$ & $\frac{2^9\cdot3^2\cdot53\cdot151\cdot263}{7^3\cdot11^2\cdot17^2}$ & $\frac{2^9\cdot3^2\cdot271897}{7^3\cdot11\cdot17^2}$ & $\frac{2^{10}\cdot3^3\cdot19\cdot38873}{7^3\cdot11^2\cdot17^2}$  \\
& & & \\[-8pt]
$x=\frac{1}{3}$ & $\frac{2^{13}\cdot17\cdot241}{7^3\cdot11^2}$ & $\frac{2^{11}\cdot3\cdot1109}{7^2\cdot11^2}$ & $\frac{2^{11}\cdot34499 }{7^3\cdot11^2}$ \\
& & & \\[-8pt]
$x=\frac{1}{18}$ & $\frac{2^7\cdot181\cdot13037}{3^2\cdot7^3\cdot11^2}$ & $\frac{2^7\cdot159671}{3\cdot7^2\cdot11^2}$ & $\frac{2^8\cdot61\cdot193\cdot211 }{3^2\cdot7^3\cdot11^2}$ \\
& & & \\[-8pt]
$x=\frac{1}{5}$ & $\frac{2^{10}\cdot3\cdot5^2\cdot19\cdot23 }{7^3\cdot11^2}$ & $\frac{2^{10}\cdot3^4\cdot5^2\cdot23 }{7^3\cdot11^2}$ & $\frac{2^{13}\cdot3\cdot5^3\cdot23 }{7^3\cdot11^2}$ \\
& & & \\[-8pt]
$x=\frac{1}{19}$ & $\frac{2^5\cdot3\cdot5\cdot139\cdot503}{7^3\cdot11^2}$ & $\frac{2^5\cdot3^3\cdot5\cdot19\cdot83}{7^2\cdot11^2}$ & $\frac{2^6\cdot3\cdot5\cdot89\cdot827}{7^3\cdot11^2}$ \\
& & & \\[-8pt]
\hline
\hline
& & & \\[-8pt]
$D_0=61$ & $D=181$ & $D=229$ & $D=1009$ \\ 
& & & \\[-8pt]
\hline
& & & \\[-8pt]
$x=\frac{1}{2}$ & $\frac{2^{10}\cdot3^2\cdot5^2\cdot19\cdot23}{7^3\cdot11^2}$ & $\frac{2^{10}\cdot3^5\cdot5^2\cdot23}{7^3\cdot11^2}$ & $\frac{2^{11}\cdot3^2\cdot5^2\cdot23\cdot211 }{7^3\cdot11^2}$ \\
& & & \\[-8pt]
$x=\frac{1}{17}$ & $\frac{2^{10}\cdot3^2\cdot5^2\cdot19\cdot23 }{7^3\cdot11^2}$ & $\frac{2^{10}\cdot3^5\cdot5^2\cdot23}{7^3\cdot11^2}$ & $\frac{2^{11}\cdot3^2\cdot5^2\cdot23\cdot211}{7^3\cdot11^2}$ \\
& & & \\[-8pt]
$x=\frac{1}{3}$ & $\frac{2^{10}\cdot3^2\cdot5^2\cdot19\cdot23}{7^3\cdot11^2}$ & $\frac{2^{10}\cdot3^5\cdot5^2\cdot23}{7^3\cdot11^2}$ & $\frac{2^{11}\cdot3^2\cdot5^2\cdot23\cdot211 }{7^3\cdot11^2}$ \\
& & & \\[-8pt]
$x=\frac{1}{18}$ & $\frac{2^{10}\cdot3^2\cdot5^2\cdot19\cdot23}{7^3\cdot11^2}$ & $\frac{2^{10}\cdot3^5\cdot5^2\cdot23}{7^3\cdot11^2}$ & $\frac{2^{11}\cdot3^2\cdot5^2\cdot23\cdot211}{7^3\cdot11^2}$ \\
& & & \\[-8pt]
$x=\frac{1}{5}$ & $\frac{2^{10}\cdot3^2\cdot5^2\cdot19\cdot23 }{7^3\cdot11^2}$ & $\frac{2^{10}\cdot3^5\cdot5^2\cdot23 }{7^3\cdot11^2}$ & $\frac{2^{11}\cdot3^2\cdot5^2\cdot23\cdot211 }{7^3\cdot11^2}$ \\
& & & \\[-8pt]
$x=\frac{1}{19}$ & $\frac{2^{10}\cdot3^2\cdot5^2\cdot19\cdot23}{7^3\cdot11^2}$ & $\frac{2^{10}\cdot3^5\cdot5^2\cdot23}{7^3\cdot11^2}$ & $\frac{2^{11}\cdot3^2\cdot5^2\cdot23\cdot211}{7^3\cdot11^2}$ \\
& & & \\
\hline
\end{tabular}
\caption{Numerical values pertaining to weight $4$ and level $15$.}
\label{table:lvl15table}
\end{table}
In particular, we observe that the values corresponding to some discriminant $D$ coincide for every $x$ listed if and only if $L\big(f_2 \otimes \big(\frac{D}{\cdot}\big),2\big)=0$.

\subsubsection{Weight \texorpdfstring{$4$}{4}, Level \texorpdfstring{$22$}{22}}
We have $\mathrm{dim}_{\C}\big(S_4(22)\big) = 7$, and consider the Hecke newform
\begin{align*}
F(z) = q - 2 q^2 + 4 q^3 + 4 q^4 + 14 q^5 - 8 q^6 - 8 q^7 - 8 q^8 - 11 q^9 + O\left(q^{10}\right),
\end{align*}
see LMFDB \cite{lmfdb} newform orbit $22$.$4$.a.b. Moreover, it has sign $+1$ under both Atkin--Lehner involutions $W_2$ and $W_{11}$. In this case, admissible discriminants are squares modulo $88$, non-squares themselves, satisfying $\big(\frac{D}{2}\big) = \big(\frac{D}{11}\big) = 1$, being coprime to $22$, and inducing a primitive Kronecker character. With Pari/GP \cite{pari}, we compute that
\begin{alignat*}{2}
& L\left(F \otimes \left(\frac{89}{\cdot}\right),2\right) = 2.956416940\ldots, \qquad
&& L\left(F \otimes \left(\frac{97}{\cdot}\right),2\right) = 1.154810097\ldots, \\
& L\left(F \otimes \left(\frac{113}{\cdot}\right),2\right) = 0.057402462\ldots, 
&& L\left(F \otimes \left(\frac{1985}{\cdot}\right),2\right) = L\left(F \otimes \left(\frac{2337}{\cdot}\right),2\right) = 0.
\end{alignat*}

Next, we show how to compute the values required in the cirterion of Theorem \ref{Thm: intro}. As the dimension of $S_4(22)$ is greater than $1$, we require the action of Hecke operators to annihilate oldspaces and newforms being linearly independent from $F$. Note that the space $S_4(2)$ is trivial, the space $S_4(11)$ is two-dimensional, and there are two normalized newforms in $S_4(22)$ being untwisted and linearly independent from $F$. We choose the Hecke polynomial
\begin{align*}
T_{13}^2-80T_{13}+400 = \left(T_{13}-20\sqrt{3}-40\right)\cdot\left(T_{13}+20\sqrt{3}-40\right)
\end{align*}
to annihilate the space $S_4(11)$ and the Hecke operators $T_3+7$ resp.\ $T_5+3$ to annihilate the other two newforms, where the primes $3$, $5$ and $13$ are coprime to the level. Consequently, our Hecke polynomial is given by
\begin{align*}
&\left(T_{13}-13^{-3}\left(20\sqrt{3}+40\right)\right)\cdot\left(T_{13}+13^{-3}\left(20\sqrt{3}-40\right)\right)\cdot\left(T_3+3^{-3}\cdot7\right)\cdot\left(T_5+5^{-3}\cdot3\right) \\
& \hspace{0.5cm} = T_{195}T_{13} + 7\cdot 3^{-2}\cdot 5^{-3}\cdot T_{13}^2 + 3\cdot 5^{-3}\cdot T_{39}T_{13} + 7\cdot 3^{-3} \cdot T_{65}T_{13} - 112\cdot 13^{-3} \cdot 15^{-2} \cdot T_{13} \\
&\hspace{1cm} - 48 \cdot 5^{-2} \cdot 13^{-3} \cdot T_{39} - 560 \cdot 39^{-3} \cdot T_{65} - 80 \cdot 13^{-3} \cdot T_{195} + 48 \cdot 5^{-1} \cdot 13^{-6} \cdot T_3 \\
&\hspace{1cm} + 2800 \cdot 3^{-3} \cdot 13^{-6} \cdot T_5 + 400 \cdot 13^{-6} \cdot T_{15} + 112\cdot 3^{-2}\cdot 5^{-1}\cdot 13^{-6}
\end{align*}
and the action of this polynomial is encoded as the function ``Heckeaction22'' in our SAGE \cite{sage} code. One example is given as follows.
\begin{lstlisting}[language=Sage, caption=Sage code for weight 4 and level 22, captionpos=b]
sage: attach('LocalPolyCodeFinal.sage')
sage: k=2
sage: N=22
sage: D0=89
sage: D=1985
sage: Heckeaction22(1/2)
4105093056512/27846975
\end{lstlisting}

The computation at other rational numbers is similar, and we obtain the values listed in Table \ref{table:lvl22table}.
\begin{table}[htbp]
\centering
\begin{tabular}{| M{1.4cm} | M{3cm} M{3cm} M{3cm} M{3cm} |}
\hline
& & & & \\[-8pt]
$D_0=89$ & $D=97$ & $D=113$ & $D=1985$ & $D=2337$ \\ 
& & & & \\[-8pt]
\hline
& & & & \\[-8pt]
$x=\frac{1}{2}$ & $\frac{2^{10}\cdot17\cdot23^2\cdot70571}{3\cdot5^2\cdot13^6}$ & $\frac{2^7\cdot1791933937 }{5^2\cdot13^6 }$ & $\frac{2^{11}\cdot7\cdot43\cdot349\cdot19081}{3\cdot5^2\cdot13^5}$ & $\frac{2^{14}\cdot7\cdot19\cdot43\cdot19081}{5^2\cdot13^5}$ \\
& & & & \\[-8pt]
$x=\frac{1}{24}$ & $\frac{2^5\cdot17\cdot10751593667}{3^3\cdot5^2\cdot13^6 }$ & $\frac{2^3\cdot210173\cdot3683233 }{3^3\cdot5^2\cdot13^6}$ & $\frac{2^{11}\cdot7\cdot43\cdot349\cdot19081}{3\cdot5^2\cdot13^5 }$ & $\frac{2^{14}\cdot7\cdot19\cdot43\cdot19081 }{5^2\cdot13^5}$ \\
& & & & \\[-8pt]
$x=\frac{1}{3}$ & $\frac{2^9\cdot17\cdot19\cdot29\cdot37\cdot10987 }{3^2\cdot5^2\cdot13^6 }$ & $\frac{2^8\cdot8063705989 }{3^2\cdot5^2\cdot13^6}$ & $\frac{2^{11}\cdot7\cdot43\cdot349\cdot19081}{3\cdot5^2\cdot13^5}$ & $\frac{2^{14}\cdot7\cdot19\cdot43\cdot19081}{5^2\cdot13^5}$ \\
& & & & \\[-8pt]
$x=\frac{1}{25}$ & $\frac{2^9\cdot17\cdot89\cdot104864951 }{3\cdot5^5\cdot13^6 }$ & $\frac{2^8\cdot19\cdot89\cdot103\cdot1929047 }{3\cdot5^5\cdot13^6 }$ & $\frac{2^{11}\cdot7\cdot43\cdot349\cdot19081 }{3\cdot5^2\cdot13^5 }$ & $\frac{2^{14}\cdot7\cdot19\cdot43\cdot19081 }{5^2\cdot13^5}$ \\
& & & & \\[-8pt]
$x=\frac{1}{5}$ & $\frac{2^9\cdot17\cdot2208991 }{3\cdot5^3\cdot13^4}$ & $\frac{2^8\cdot1033808767 }{3\cdot5^3\cdot13^5 }$ & $\frac{2^{11}\cdot7\cdot43\cdot349\cdot19081 }{3\cdot5^2\cdot13^5 }$ & $\frac{2^{14}\cdot7\cdot19\cdot43\cdot19081 }{5^2\cdot13^5 }$ \\
& & & & \\[-8pt]
$x=\frac{1}{27}$ & $\frac{2^9\cdot17\cdot3739\cdot4852451 }{3^6\cdot5^2\cdot13^6 }$ & $\frac{2^8\cdot29\cdot211\cdot1667\cdot64033 }{3^6\cdot5^2\cdot13^6 }$ & $\frac{2^{11}\cdot7\cdot43\cdot349\cdot19081 }{3\cdot5^2\cdot13^5 }$ & $\frac{2^{14}\cdot7\cdot19\cdot43\cdot19081 }{5^2\cdot13^5 }$ \\
& & & & \\
\hline
\end{tabular}
\caption{Numerical values pertaining to weight $4$ and level $22$.}
\label{table:lvl22table}
\end{table}
In particular, we observe that the values corresponding to some discriminant $D$ coincide for every $x$ listed if and only if $L\big(F \otimes \big(\frac{D}{\cdot}\big),2\big)=0$.

\section{Questions for future work}
\label{FutureWork}
We conclude with a few questions for future work.

\begin{enumerate}[leftmargin=*]
\item Skoruppa \cite{skoruppa} used skew-holomorphic Jacobi forms to obtain a similar condition to Theorem \ref{MainCongNumThm} for fundamental discriminants $D \equiv 1\pmod{8}$. One can use locally harmonic Maass forms with a modification of the genus character to prove an equivalent formula. To the authors' knowledge, Skoruppa's method crucially uses that $D$ is a quadratic residue modulo $8$. It would be interesting to extend his theory to other discriminants with a suitable generalization or to tie that theory directly to the local polynomials in this paper.
\item It would be interesting to try to numerically optimize our methods discussed in Section~\ref{Sec: examples}. We have not attempted to do so here. Further, the use of Atkin--Lehner involutions as well as Hecke operators may dramatically speed up computations. See for example \cite{ComputingMFs} for a discussion of computing for the LMFDB \cite{lmfdb} using Atkin--Lehner eigenspaces. Example 4.5.1 therein describes a decomposition of a 159-dimensional space of cusp forms into eigenspaces with dimensions ranging from 1 to 29. The advantages of using Atkin--Lehner involutions would be that they involve plugging in fewer rational points, and many fewer operators would be required.
\item It would be interesting to further explore the arithmetic consequences of our formulas. For instance, is there a higher-dimensional analogue of Monsky-style results such as that at the end of Section~\ref{PreviousWorkSection}? Is there a general interpretation using hyperbolic geometry (as Genz used) or Selmer groups that can ``witness'' vanishing or non-vanishing of twisted central $L$-values?
\end{enumerate}

\end{document}